\documentclass[a4paper,12pt]{amsart}
\usepackage{amsmath}
\usepackage{amsmath}
\usepackage{amssymb}
\usepackage{amscd}
\usepackage{mathrsfs}
\usepackage[all]{xy}
\usepackage{yhmath}

\def\mathcal{\mathscr}
\everymath{\displaystyle}
\newtheorem{thm}{Theorem}[section]
\newtheorem{lem}[thm]{Lemma}
\newtheorem{cor}[thm]{Corollary}
\newtheorem{prop}[thm]{Proposition}
\newtheorem{conj}[thm]{Conjecture}

\theoremstyle{definition}

\newtheorem{rem}[thm]{Remark}

\newtheorem{defn}[thm]{Definition}

\newtheorem{ex}[thm]{Example}

\newcommand{\mca}[1]{{\mathcal{#1}}}

\def\Q{{\mathbb Q}}
\def\Z{{\mathbb Z}}
\def\C{{\mathbb C}}
\def\R{{\mathbb R}}

\def\ad{\text{\rm ad}\,}

\def\CF{\text{\rm CF}\,}
\def\CM{\text{\rm CM}\,}
\def\CZ{\text{\rm CZ}\,}
\def\conv{\text{\rm conv}\,}

\def\Diff{\text{\rm Diff}\,}
\def\diam{\text{\rm diam}}

\def\dist{\text{\rm dist}}

\def\EHZ{\text{\rm EHZ}}

\def\ev{\text{\rm ev}\,}
\def\ep{\varepsilon} 

\def\FHW{\text{\rm FHW}}

\def\HF{\text{\rm HF}\,}
\def\HM{\text{\rm HM}\,}
\def\HZ{\text{\rm HZ}\,}

\def\id{\text{\rm id}\,}
\def\ind{\text{\rm ind}\,}

\def\interior{\text{\rm int}\,}

\def\len{\text{\rm len}}

\def\Morse{\text{\rm Morse}}

\def\ph{\varphi} 
\def\pr{\text{\rm pr}}

\def\std{\text{\rm std}\,}
\def\supp{\text{\rm supp}\,}
\def\SH{\text{\rm SH}\,}

\begin{document}
\pagestyle{plain}
\thispagestyle{plain}

\title[Symplectic homology of fiberwise convex sets and homology of loop spaces]
{Symplectic homology of fiberwise convex sets and homology of loop spaces} 

\author[Kei Irie]{Kei Irie}
\address{Research Institute for Mathematical Sciences, Kyoto University, Kyoto 606-8502, JAPAN} 
\email{iriek@kurims.kyoto-u.ac.jp} 
\date{\today}

\begin{abstract} 
For any nonempty, compact and fiberwise convex set $K$ in $T^*\R^n$, 
we prove an isomorphism between symplectic homology of $K$ and a certain relative homology of loop spaces of $\R^n$. 
We also prove a formula which computes symplectic homology capacity (which is a symplectic capacity defined from symplectic homology) 
of $K$ using homology of loop spaces. 
As applications, we prove 
(i) symplectic homology capacity of any convex body is equal to its Ekeland-Hofer-Zehnder capacity, 
(ii) a certain subadditivity property of the Hofer-Zehnder capacity, which is a generalization of a result previously proved by Haim-Kislev. 
\end{abstract}

\maketitle

\section{Introduction} 

\subsection{Symplectic homology and the capacity $c_{\SH}$}

Let $n$ be a positive integer. 
Let us consider coordinates $q_1, \ldots, q_n, p_1, \ldots, p_n$ on $T^*\R^n$, where 
$q_1, \ldots, q_n$ are coordinates on $\R^n$ and $p_1, \ldots, p_n$ are coordinates on fibers with respect to the global frame $dq_1, \ldots, dq_n$. 
We often abbreviate $(q_1, \ldots, q_n)$ by $q$ and $(p_1, \ldots, p_n)$ by $p$. 
Let $\omega_n:= \sum_{1 \le i \le n} dp_i \, dq_i \in \Omega^2(T^*\R^n)$. 
For any nonempty compact set $K \subset T^*\R^n$ and real numbers $a<b$, 
one can define a $\Z$-graded $\Z/2$-vector space 
$\SH_*^{[a,b)}(K)$, which is called \textit{symplectic homology}
(see Section 2.2 for details). 

\begin{rem} 
Throughout this paper, all (co)homology groups are defined over $\Z/2\Z$, unless otherwise specified. 
\end{rem} 

When $K$ satisfies certain nice conditions, we say that $K$ is a \textit{restricted contact type} (RCT) set 
(see Definition \ref{defn_RCT_set}; note that our definition of RCT sets is slightly more generalized than the usual definition). 
Any compact star-shaped (in particular, convex) set is a RCT set (Lemma \ref{starshaped_implies_RCT}). 
For any RCT set $K \subset T^*\R^n$ and $a \in \R_{>0}$, 
there exists a natural linear map 
\[
i^a_K: H_{*+n}(T^*\R^n, T^*\R^n \setminus K)  \to \SH^{[0,a)}_*(K). 
\] 
See Section 2.3 for the definition of $i^a_K$. 
Also, as we define in Section 2.4, 
there exists a canonical element $\nu^{T^*\R^n}_K \in H_{2n}(T^*\R^n, T^*\R^n \setminus K)$. 
Then let us define the following numerical invariant: 
\[ 
c_{\SH}(K):= \inf\{ a \in \R_{>0} \mid i^a_K (\nu^{T^*\R^n}_K)= 0\}. 
\] 
In this paper, the invariant $c_{\SH}$ is called \textit{symplectic homology capacity}. 

\begin{rem}
The first symplectic capacity defined from symplectic homology was introduced by Floer-Hofer-Wysocki \cite{Floer_Hofer_Wysocki}, 
who defined a capacity (denoted by $c_{\FHW}$) for arbitrary open sets in the symplectic vector space. 
The above definition of $c_{\SH}$ is due to Hermann \cite{Hermann},
which is based on the idea by Viterbo \cite{Viterbo_GAFA} (see Section 5.3 of \cite{Viterbo_GAFA}). 
Indeed, Hermann (Proposition 5.7 of \cite{Hermann}) proved that 
(in the language of the present paper) 
any $C^\infty$-RCT set $K$ (see Definition \ref{defn_RCT_set})
satisfies $c_{\SH}(K) = c_{\FHW}(\interior(K))$. 
Here $\interior(K)$ denotes the interior of $K$. 
\end{rem} 

Although symplectic homology and the capacity $c_{\SH}$ are fundamental quantitative invariants 
of subsets of the symplectic vector space, 
they are notoriously difficult to compute, or even to estimate. 
This is because symplectic homology is a version of Floer homology, 
whose definition involves counting solutions of nonlinear PDEs (so called Floer equations), 
thus it is very difficult to compute these invariants directly from definitions. 
The core results of this paper, which we discuss in Section 1.2, 
enable us to investigate these invariants via computations of homology of loop spaces. 

\subsection{Main results} 

The core results of this paper are Theorem \ref{SH_and_loop_space_homology} and Corollary \ref{c_SH_and_loop_space_homology}. 
Corollary \ref{c_SH_and_loop_space_homology} has two applications: 
Theorem \ref{SH=EHZ} and Theorem \ref{HZ_subadditivity}. The goal of this subsection is to describe these four results. 

Theorem \ref{SH_and_loop_space_homology} shows that, 
for any nonempty compact set $K \subset T^*\R^n$ 
which is fiberwise convex (i.e. $K \cap T_q^*\R^n$ is convex for every $q \in \R^n$), 
symplectic homology of $K$ is isomorphic to a certain relative homology of loop spaces of $\R^n$. 
Theorem \ref{SH_and_loop_space_homology} is a version of the well-known isomorphism between Floer homology of cotangent bundles and homology of loop spaces. 
Indeed, the proof of Theorem \ref{SH_and_loop_space_homology} heavily relies on the proof by Abbondandolo-Schwarz \cite{Abbondandolo_Schwarz_Floer_homology}
of this isomorphism. 

Corollary \ref{c_SH_and_loop_space_homology}, 
which is an easy consequence of Theorem \ref{SH_and_loop_space_homology}, 
shows that if $K$ is a RCT set then $c_{\SH}(K)$ is equal to a certain min-max value defined from homology of loop spaces. 
In the rest of this subsection, we present two applications of Corollary \ref{c_SH_and_loop_space_homology}: 
Theorem \ref{SH=EHZ} and Theorem \ref{HZ_subadditivity}. 

To state Theorem \ref{SH=EHZ}, let us recall the definition of the Ekeland-Hofer-Zehnder capacity (which we denoted by $c_{\EHZ}$) of convex bodies. 
For definitions of ``symplectic action'' and ``closed characteristics'', see Section 2.3. 

\begin{defn}\label{defn_EHZ} 
$K \subset T^*\R^n$ is called a \textit{convex body} if $K$ is 
compact, convex, and $\interior (K) \ne \emptyset$. 
When $\partial K$ is a $C^\infty$-hypersurface, 
then its Ekeland-Hofer-Zehnder capacity $c_{\EHZ}(K)$ is defined as the
minimum symplectic action of closed characteristics on $\partial K$. 
For arbitrary convex body $K$, we define 
\[ 
c_{\EHZ}(K):= \inf\{ c_{\EHZ}(K') \mid \text{$K'$ is a convex body with $C^\infty$-boundary such that $K \subset K'$} \}. 
\] 
\end{defn} 
Now let us state our first application of Corollary \ref{c_SH_and_loop_space_homology}: 

\begin{thm}\label{SH=EHZ} 
$c_{\SH}(K)=c_{\EHZ}(K)$ for any convex body $K \subset T^*\R^n$. 
\end{thm} 

\begin{rem} 
\begin{itemize}
\item 
Theorem \ref{SH=EHZ} is also proved by Abbondandolo-Kang \cite{Abbondandolo_Kang}. 
Their proof is based on an isomorphism (which is the main result of \cite{Abbondandolo_Kang})
between 
the filtered Floer complex of a convex quadratic Hamiltonian on $T^*\R^n$ (satisfying some technical conditions) 
and 
the filtered Morse complex of its Clarke dual action functional. 
\item 
Using $S^1$-equivairiant symplectic homology, 
one can define a sequence of capacities $(c^k_{\SH^{S^1}})_{k \ge 1}$. 
Felix Schlenk \cite{Schlenk} pointed out 
that, 
assuming some standard properties of these capacities, 
Theorem \ref{SH=EHZ} implies $c^1_{\SH^{S^1}}(K)=c_{\SH}(K)$ for any convex body $K \subset T^*\R^n$; 
see Section 2.5 for details. 
\end{itemize} 
\end{rem} 

Theorem \ref{SH=EHZ} is motivated 
by the following folk conjecture, which says that all symplectic capacities on $T^*\R^n$ coincide for convex bodies
(see Section 5 of \cite{Ostrover_ICM} and the references therein): 

\begin{conj}\label{Strong_Viterbo_Conjecture} 
Let $c$ be any symplectic capacity on $T^*\R^n$; 
namely, $c$ is a map from the set of all subsets of $T^*\R^n$
to $[0,\infty]$ 
which satisfies the following three properties: 
\begin{itemize} 
\item For any $S \subset T \subset T^*\R^n$, there holds $c(S) \le c(T)$. 
\item For any $S \subset T^*\R^n$, $a \in \R_{>0}$ and $\ph \in \Diff(T^*\R^n)$ such that $\ph^*\omega_n = a \omega_n$, there holds 
$c(\ph(S)) = a c(S)$. 
\item $c(\{(q,p)\in T^*\R^n\mid |q|^2+|p|^2\le 1\})=c(\{(q,p)\in T^*\R^n\mid q_1^2+p_1^2\le 1\}) = \pi$. 
\end{itemize} 
Then $c(K) = c_{\EHZ}(K)$ for any convex body $K$. 
\end{conj} 

Conjecture \ref{Strong_Viterbo_Conjecture} 
is still widely open. 
As far as the author knows, 
Conjecture \ref{Strong_Viterbo_Conjecture} 
was verified only for the first equivariant Ekeland-Hofer capacity and the Hofer-Zehnder capacity. 
The result for the first equivariant Ekeland-Hofer capacity was mentioned by Viterbo (Proposition 3.10 of \cite{Viterbo_capacity}), 
and a detailed proof can be found in Section 6 of Gutt-Hutchings-Ramos \cite{Gutt_Hutchings_Ramos}. 
The result on the Hofer-Zehnder capacity is due to Hofer-Zehnder \cite{Hofer_Zehnder_convex}. 
Theorem \ref{SH=EHZ} verifies Conjecture \ref{Strong_Viterbo_Conjecture} for the symplectic homology capacity $c_{\SH}$. 

Our second application of Corollary \ref{c_SH_and_loop_space_homology} is a certain subadditivity property of the Hofer-Zehnder capacity. 
Let us recall the definition of the Hofer-Zehnder capacity: 

\begin{defn} 
$H \in C_c^\infty(T^*\R^n, \R_{\ge 0})$ is called \textit{Hofer-Zehnder admissible} if 
there exists a nonempty open set $U \subset T^*\R^n$ such that $H|_U \equiv \max H$, 
and every nonconstant periodic orbit of its Hamiltonian vector field $X_H$ (see the first paragraph of Section 2 for our convention) has period strictly larger than $1$.
Let $\mca{H}_\ad$ denote the set of all Hofer-Zehnder admissible functions on $(T^*\R^n, \omega_n)$. 
For any $S \subset T^*\R^n$ such that $\interior (S) \ne \emptyset$, 
 its Hofer-Zehnder capacity $c_{\HZ}(S) \in \R_{>0}$ is defined as 
\[ 
c_{\HZ}(S): = \sup \{ \max H \mid  H \in  \mca{H}_\ad, \, \supp H \subset S \}. 
\]
\end{defn} 

Now we can state our second application of Corollary \ref{c_SH_and_loop_space_homology}: 

\begin{thm}\label{HZ_subadditivity} 
Let $K$ be any compact set in $T^*\R^n$ with $\interior (K) \ne \emptyset$, 
and $\Pi$ be any hyperplane in $T^*\R^n$ which intersects $\interior (K)$. 
Let $\Pi^+$ and $\Pi^-$ be distinct closed halfspaces such that $\partial \Pi^+ =\partial \Pi^- = \Pi$.  
Then, setting $K^+:= K \cap \Pi^+$ and $K^-:= K \cap \Pi^-$, there holds 
\[ 
c_{\HZ}(K) \le c_{\EHZ}(\conv(K^+)) + c_{\EHZ}(\conv(K^-)),
\] 
where $\conv$ denotes the convex hull. 
\end{thm} 

Theorem \ref{HZ_subadditivity} can be rephrased as follows: 
for any $K$ and $\Pi$ such that $K^+$ and $K^-$ are convex, 
$c_{\HZ}(K) \le c_{\EHZ}(K^+) + c_{\EHZ}(K^-)$. 
In particular, we recover the following result by Haim-Kislev \cite{Haim_Kislev} as a corollary: 

\begin{cor}[\cite{Haim_Kislev} Theorem 1.8]\label{Haim_Kislev} 
Let $K$ be any convex body in $T^*\R^n$
and $\Pi$ be any hyperplane in $T^*\R^n$ which intersects $\interior (K)$. 
Then, $c_{\EHZ}(K) \le c_{\EHZ}(K^+) + c_{\EHZ}(K^-)$. 
\end{cor} 

The proof in \cite{Haim_Kislev} uses a combinatorial formula (Theorem 1.1 of \cite{Haim_Kislev}) which computes the EHZ capacity of convex polytopes, 
and it seems difficult to extend this proof to prove Theorem \ref{HZ_subadditivity} when $K$ is not convex. 

Theorem \ref{HZ_subadditivity} is inspired by the following conjecture 
by Akopyan-Karasev-Petrov \cite{Akopyan_Karasev_Petrov}: 

\begin{conj}[\cite{Akopyan_Karasev_Petrov}]\label{AKP}
Let $K, K_1, \ldots, K_m$ be convex bodies in $T^*\R^n$. 
If $K \subset \bigcup_{i=1}^m K_i$, then $c_{\EHZ}(K) \le \sum_{i=1}^m c_{\EHZ}(K_i)$. 
\end{conj} 

In \cite{Akopyan_Karasev_Petrov}, Conjecture \ref{AKP} was verified for hyperplane cuts of round balls, 
which was later generalized to hyperplane cuts of arbitrary convex bodies (Corollary \ref{Haim_Kislev}). 
Note that the convexity of $K_1, \ldots, K_m$ is essential in Conjecture \ref{AKP}, 
as shown by examples in Section 5.1 of \cite{Akopyan_Karasev_Petrov}, 
for which the subadditivity fails without the convexity assumption. 
Let us also mention the following Proposition \ref{2021_0530_1}, which gives another such example. 
The proof of Proposition \ref{2021_0530_1}, which we explain in Section 7, is elementary. 

\begin{prop}\label{2021_0530_1} 
Let $n \ge 2$ be an integer. 
For any bounded $B \subset T^*\R^n$ and any $\ep \in \R_{>0}$, 
there are compact star-shaped sets $K_1, K_2 \subset T^*\R^n$ 
such that $B \subset K_1 \cup K_2$ and $e(K_1), e(K_2)< \ep$, 
where $e$ denotes the Hamiltonian displacement energy. 
\end{prop} 

On the other hand, it seems unknown if the following conjecture, which is stronger than Conjecture \ref{AKP}, 
holds true. 

\begin{conj}\label{stronger_subadditivity} 
For any convex bodies $K_1, \ldots, K_m$ in $T^*\R^n$, 
\[
c_{\HZ} \big( \bigcup_{i=1}^m K_i \big) \le \sum_{i=1}^m c_{\EHZ}(K_i). 
\] 
\end{conj} 

As far as the author knows,
Theorem \ref{HZ_subadditivity} 
is the first verification of Conjecture \ref{stronger_subadditivity} 
in a situation not covered by Conjecture \ref{AKP}. 

\subsection{Structure of this paper} 
Let us explain the structure of this paper. 
In Section 2 we review basics of symplectic homology. 
In particular, we recall the definition of the capacity $c_{\SH}$ and explain its basic properties. 
In Section 3, we state Theorem \ref{SH_and_loop_space_homology}, 
and deduce Corollary \ref{c_SH_and_loop_space_homology} from Theorem \ref{SH_and_loop_space_homology}. 
Section 4 is devoted to the proof of Theorem \ref{SH_and_loop_space_homology}, 
which is based on the ``hybrid moduli space'' method of Abbondandolo-Schwarz \cite{Abbondandolo_Schwarz_Floer_homology}. 
The outline of the proof is sketched in the first paragraph of Section 4. 
Section 4 is the most technical section, and can be skipped at the first reading. 
In Section 5 we prove Theorem \ref{SH=EHZ}, and in Section 6 we prove Theorem \ref{HZ_subadditivity}. 
Using Corollary \ref{c_SH_and_loop_space_homology}, 
these results can be proved by elementary arguments about loop spaces. 
In particular, the key estimate is Lemma \ref{upper_bound_of_length}.
In Section 7, we prove Proposition \ref{2021_0530_1}. 
This section can be read independently from Sections 2--6. 

\textbf{Acknowledgement.} 
The author thanks Felix Schlenk for pointing out an application discussed in Section 2.5, 
and his comments on an earlier version of this paper. 
The author also thanks Alberto Abbondandolo and Jungsoo Kang for sharing their manuscript \cite{Abbondandolo_Kang}
and having discussions about relations between their approach and the author's. 
Finally, the author thanks the referee for many comments which are very helpful to improve readability of this paper. 
This research is supported by JSPS KAKENHI Grant No.18K13407 and No.19H00636. 

\section{Symplectic homology and the capacity $c_{\SH}$}

For any $h \in C^\infty(T^*\R^n)$, its Hamiltonian vector field $X_h \in \mca{X}(T^*\R^n)$ is defined by 
$\omega_n(X_h, \, \cdot \,) = -dh(\, \cdot \,)$. 
Let $S^1:= \R/\Z$. 
For any $H \in C^\infty(S^1 \times T^*\R^n)$ and $t \in S^1$, 
we define $H_t \in C^\infty(T^*\R^n)$ by $H_t(q,p):= H(t, q, p)$. 
Let 
\[ 
\mca{P}(H):= \{ \gamma: S^1 \to T^*\R^n \mid \dot{\gamma}(t) = X_{H_t}(\gamma(t)) \, (\forall t \in S^1) \}. 
\] 
$\gamma \in \mca{P}(H)$ is called \textit{nondegenerate} if 
$1$ is not an eigenvalue of 
$(d\ph^1_H)_{\gamma(0)}$, 
where $(\ph^t_H)_{0 \le t \le 1}$ denotes the Hamiltonian isotopy generated by $H$. 

\begin{rem}
The isotopy $(\ph^t_H)_{0 \le t \le 1}$ may not be globally defined, 
but it is defined at least on a neighborhood of $\gamma(0)$. 
\end{rem}

\subsection{Filtered Floer homology} 

In this subsection, we review basic facts about filtered Floer homology of 
(time-dependent) Hamiltonians on $\C^n$ which are compact perturbations of quadratic functions. 
The results in this subsection are essentially contained in \cite{Floer_Hofer}. 
However, here we mainly follow \cite{JSG}, 
since the class of Hamiltonians we consider is slightly different from that in \cite{Floer_Hofer}. 

For any $H \in C^\infty(S^1 \times T^*\R^n)$
we consider the following conditions: 
\begin{itemize} 
\item[(H0):] Every $\gamma \in \mca{P}(H)$ is nondegenerate. 
\item[(H1):] There exist $A \in \R_{>0} \setminus \pi \Z$ and $B \in \R$ such that the function 
\[
H(t, q,p) - A(|q|^2+|p|^2) - B \in C^\infty(S^1 \times T^*\R^n) 
\]
is compactly supported. 
\end{itemize} 

In the following we assume that $H \in C^\infty(S^1 \times T^*\R^n)$ satisfies (H0) and (H1). 
Note that (H1) implies that all elements of $\mca{P}(H)$ are contained in a compact subset of $T^*\R^n$. 
This is because on the complement of a sufficiently large compact set, 
every orbit of $X_H$ is periodic with the minimal period equal to $\frac{\pi}{A}$. 
By $A \notin \pi \Z$, there exists no periodic orbit with period $1$ on the complement. 
Moreover (H0) implies that $\mca{P}(H)$ is discrete, thus it is finite. 

For any real numbers $a<b$ and 
$k \in \Z$, let 
$\CF^{[a,b)}_k(H)$ denote the $\Z/2$-vector space 
spanned by 
\[ 
\{ \gamma \in \mca{P}(H) \mid \mca{A}_H(\gamma) \in [a,b),\,\ind_{\CZ}(\gamma) = k \}. 
\] 
Here, $\ind_{\CZ}$ denotes the Conley-Zehnder index (see Section 1.3 of \cite{Floer_Hofer})
and $\mca{A}_H$ is defined by 
\[ 
\mca{A}_H(\gamma):= \int_{S^1} \gamma^*\bigg(\sum_i p_i dq_i \bigg) - H_t(\gamma(t)) \, dt. 
\] 

To define a boundary operator on $\CF^{[a,b)}_*(H)$, 
we take $J=(J_t)_{t \in S^1}$, which is a $C^\infty$-family of almost complex structures on $T^*\R^n$
with the following condition:
\begin{itemize} 
\item[(J1):] For every $t\in S^1$, $J_t$ is compatible with respect to $\omega_n$. 
Namely, 
$g_{J_t}(v,w) := \omega_n(v, J_t w)$ is a Riemannian metric on $T^*\R^n$. 
\end{itemize} 
For any $J$ satisfying (J1) 
and $x_-, x_+ \in \mca{P}(H)$, 
we define 
\begin{align*} 
\mca{M}_{H,J}(x_-, x_+)&:= \{ u: \R \times S^1 \to T^*\R^n \mid \partial_s u - J_t(\partial_t u - X_{H_t}(u)) = 0, \\
& \lim_{s \to \pm \infty} u_s = x_\pm \}. 
\end{align*} 
Here $s$ denotes the coordinate on $\R$, 
$t$ denotes the coordinate on $S^1$, 
and $u_s:S^1 \to T^*\R^n$ is defined by $u_s(t):=u(s,t)$. 
We set 
$\bar{\mca{M}}_{H,J}(x_-, x_+):= \mca{M}_{H,J}(x_-, x_+)/\R$, 
where the $\R$ action on $\mca{M}_{H, J}(x_-, x_+)$ 
is defined by 
\[
(r \cdot u)(s,t):=u(s-r, t) \qquad (u \in \mca{M}_{H,J}(x_-, x_+), \, r \in \R).
\]

Let us define the \textit{standard complex structure} on $T^*\R^n$, 
which is denoted by $J_\std$, by 
\[ 
J_\std(\partial_{p_i}) = \partial_{q_i}, \qquad
J_\std(\partial_{q_i}) = - \partial_{p_i} \qquad
(1 \le i \le n). 
\] 

\begin{lem}\label{C0_HF_1} 
Suppose $H$ satisfies (H0) and (H1), 
$J$ satisfies (J1), 
and $\sup_{t \in S^1} \|J_t-J_\std\|_{C^0}$ is sufficiently small. 
Then 
$\sup_{\substack{x_-,x_+\in \mca{P}(H)\\ u\in\mca{M}_{H,J}(x_-,x_+) \\ (s,t)\in\R\times S^1}} |u(s,t)| <\infty.$
\end{lem} 
\begin{proof}
This lemma follows from Lemma 2.3 in \cite{JSG}; 
note that conditions (H0), (J1) in \cite{JSG} are the same as (H0), (J1) in this paper, 
and the condition (H1) in \cite{JSG} is weaker than (H1) in this paper. 
\end{proof} 

For a generic (with respect to the $C^\infty$-topology) choice of $J$, 
the moduli space 
$\bar{\mca{M}}_{H, J}(x_-, x_+)$ is cut out transversally for any pair $(x_-, x_+)$. 
For any such $J$, 
$\bar{\mca{M}}_{H,J}(x_-, x_+)$ is a finite set if $\ind_\CZ(x_+) = \ind_\CZ(x_-)-1$, 
and the linear map 
\[ 
\partial_{H,J}: \CF^{[a,b)}_*(H) \to \CF^{[a,b)}_{*-1}(H); \quad x_- \mapsto \sum_{\ind_\CZ(x_+)=\ind_\CZ(x_-)-1}  \#_2 \bar{\mca{M}}_{H,J}(x_-, x_+) \cdot  x_+ 
\]
satisfies $\partial_{H,J}^2=0$, where $\#_2$ denotes the cardinality modulo $2$. 
The homology of the chain complex $(\CF^{[a,b)}_*(H), \partial_{H, J})$ does not depend on the choice of $J$. 
This homology is 
denoted by $\HF^{[a,b)}_*(H)$ and called
\textit{filtered Floer homology} of $H$. 
For any $a, b, a', b' \in \R$ 
with $a<b$, $a'<b'$, $a \le a'$ and $b \le b'$, 
one can define a natural linear map 
$\HF^{[a,b)}_*(H) \to \HF^{[a', b')}_*(H)$. 

\begin{rem} 
As we remarked at the beginning of this subsection, 
the fact $\partial_{H,J}^2=0$, 
as well as the independence of the homology on the choice of $J$, 
are due to \cite{Floer_Hofer} and references therein. 
\end{rem} 

Suppose that $H^-, H^+ \in C^\infty(S^1 \times T^*\R^n)$ satisfy (H0), (H1) and 
\begin{equation}\label{H0H1}
H^-(t,q,p)< H^+(t,q,p) \qquad (\forall (t,q,p) \in S^1 \times T^*\R^n). 
\end{equation} 
Then, for any real numbers $a<b$ one can define a linear map (called \textit{monotonicity map}) 
\[
\HF_*^{[a,b)}(H^-) \to \HF_*^{[a,b)}(H^+)
\]
as follows. 
First, we take $J^- = (J^-_t)_{t\in S^1}$ and $J^+ = (J^+_t)_{t\in S^1}$ such that 
$J^-$ defines a boundary map on $\CF_*(H^-)$ and 
$J^+$ defines a boundary map on $\CF_*(H^+)$. 
Next, we take 
a $C^\infty$-family of Hamiltonians $H=(H_{s,t})_{(s,t) \in \R \times S^1}$ 
and 
a $C^\infty$-family of almost complex structures $J=(J_{s,t})_{(s,t) \in \R \times S^1}$ 
such that the following conditions hold: 

\begin{enumerate}
\item[(HH1):] There exists $s_0>0$ such that 
$H_{s,t}(q,p) =\begin{cases} H^-(t, q, p) &(s \le -s_0) \\ H^+(t,q,p) &( s \ge  -s_0). \end{cases}$
\item[(HH2):] $\partial_s H_{s,t}(q,p) \ge 0$ for any $(s,t,q,p) \in \R \times S^1 \times T^*\R^n$. 
\item[(HH3):] There exist $a(s), b(s) \in C^\infty(\R)$ such that the following conditions hold: 
\begin{itemize}
\item $a'(s) \ge 0$ for any $s$. 
\item $a(s) \in \pi \Z \implies a'(s) >0$. 
\item Setting $\Delta_{s,t}(q,p):=H(s,t,q,p)- a(s)(|q|^2+|p|^2)-b(s)$, there holds 
\[
\sup_{(s,t)} \| \Delta_{s,t} \|_{C^1(T^*\R^n)} < \infty, \qquad
\sup_{(s,t)} \| \partial_s \Delta_{s,t} \|_{C^0(T^*\R^n)} < \infty. 
\]
\end{itemize}
\item[(JJ1):] There exists $s_1>0$ such that 
$J_{s,t}= \begin{cases} J^-_t &(s \le -s_1) \\ J^+_t   &(s \ge s_1). \end{cases}$
\item[(JJ2):] For every $(s,t) \in \R \times S^1$, $J_{s,t}$ is compatible with $\omega_n$. 
\end{enumerate}

\begin{rem} 
For any $H^-$ and $H^+$
satisfying (H0), (H1) and (\ref{H0H1}), 
there exists $H=(H_{s,t})_{(s,t) \in \R \times S^1}$ satisfying (HH1), (HH2) and (HH3), 
as we explained in pp.517 of \cite{JSG}. 
Let us repeat the explanation for the convenience of the reader. 
Take $\rho \in C^\infty(\R)$ such that 
$\rho|_{\R_{\le 0}} \equiv 0$, 
$\rho|_{\R_{\ge 1}} \equiv 1$ and 
$0 < \rho(s)<1$, $\rho'(s)>0$ for any $0<s<1$. 
Then let us define $H=(H_{s,t})_{(s,t) \in \R \times S^1}$ by 
\[ 
H_{s,t}(q,p):= (1-\rho(s))  H^-(t,q,p) + \rho(s)  H^+(t,q,p). 
\] 
On the other hand, the existence of $J=(J_{s,t})_{(s,t) \in \R \times S^1}$
satisfying (JJ1) and (JJ2)
is straightforward from the fact that the set of almost complex structures compatible with $\omega_n$ is contractible. 
\end{rem} 

For any $H = (H_{s,t})_{(s,t) \in \R \times S^1}$ and $J =(J_{s,t})_{(s,t) \in \R \times S^1}$ 
satisfying the above conditions, and for 
any $x_- \in \mca{P}(H^-)$ and $x_+ \in  \mca{P}(H^+)$, 
we consider the moduli space 
\[ 
\mca{M}_{H, J}(x_-, x_+):= \{ u: \R \times S^1 \to T^*\R^n \mid \partial_s u - J_{s,t}(\partial_t u - X_{H_{s,t}}(u)) = 0, \, \lim_{s \to \pm \infty} u_s = x_{\pm}\}. 
\] 

\begin{lem}\label{C0_HF_2} 
Suppose that $H$ satisfies (HH1), (HH2) and (HH3). 
If 
$J$ satisfies (JJ1), (JJ2) and 
$\sup_{(s,t) \in \R\times S^1} \|J_{s,t}-J_\std\|_{C^0}$ is sufficiently small, 
then 
\[
\sup_{\substack{x_-\in\mca{P}(H^-), x_+\in\mca{P}(H^+)\\u\in\mca{M}_{H,J}(x_-,x_+)\\(s,t)\in\R\times S^1}}|u(s,t)|<\infty.
\]
\end{lem} 
\begin{proof}
See Lemma 2.4 in \cite{JSG}. 
\end{proof} 

For a generic choice of $(H,J)$ which satisfies 
the assumptions in Lemma \ref{C0_HF_2}, 
$\mca{M}_{H, J}(x_-, x_+)$ is cut out transversally for any pair $(x_-, x_+)$. 
In particular, 
$\mca{M}_{H,J}(x_-, x_+)$ is a finite set if $\ind_\CZ(x_+) = \ind_\CZ(x_-)$, 
and the linear map 
\[ 
\Phi:\CF_*^{[a,b)}(H^-)\to \CF_*^{[a,b)}(H^+) ; \, x_-  \mapsto \sum_{\ind_\CZ(x_+)=\ind_\CZ(x_-)}  \#_2 \mca{M}_{H, J}(x_-, x_+) \cdot x_+ 
\] 
satisfies 
$\partial_{H^+,J^+}\circ \Phi =\Phi \circ \partial_{H^-,J^-}$. 
The induced map on homology 
\[ 
H_*(\Phi): \HF^{[a,b)}_*(H^-) \to \HF^{[a,b)}_*(H^+)
\]
does not depend on the choice of $(H, J)$; see Section 4.3 of \cite{Floer_Hofer}. 
This completes the definition of the monotonicity map. 

For any $H^0, H^1, H^2 \in C^\infty(S^1 \times T^*\R^n)$ 
satisfying (H0), (H1) and 
\[
H^0(t,q,p)<H^1(t,q,p)<H^2(t,q,p) \qquad( \forall (t,q,p) \in S^1 \times T^*\R^n), 
\] 
the diagram 
\[ 
\xymatrix{
\HF^{[a,b)}_*(H^0)\ar[rr]\ar[rd]&& \HF^{[a,b)}_*(H^2) \\
&\HF^{[a,b)}_*(H^1)\ar[ru]&
}
\]
commutes (all three maps are monotonicity maps). 

\subsection{Symplectic homology} 

For any nonempty compact set $K$ in $T^*\R^n$, 
let $\mca{H}_K$ denote the set of $H \in C^\infty(S^1 \times T^*\R^n)$ 
which satisfies (H0), (H1) and 
$H(t,q,p)<0$ for any $(t,q,p) \in S^1 \times K$. 
Then $\mca{H}_K$ becomes a directed set by setting 
$H^0 < H^1$ if and only if $H^0(t,q,p) < H^1(t,q,p)$ for any $(t,q,p) \in S^1 \times T^*\R^n$. 
For any real numbers $a<b$, we set 
\[ 
\SH^{[a,b)}_*(K) := \varinjlim_{H \in \mca{H}_K} \HF^{[a,b)}_*(H), 
\] 
where the limit is taken by monotonicity maps. 

For any $a, b, a', b' \in \R$ with 
$a<b$, $a'<b'$, $a \le a'$, $b \le b'$, 
and nonempty compact sets $K' \subset K$, 
one can define a natural linear map 
$\SH^{[a,b)}_*(K) \to \SH^{[a', b')}_*(K')$. 
Also, for any $c \in \R_{>0}$ one can define a natural isomorphism 
\[ 
\SH^{[a,b)}_*(K) \cong \SH^{[c^2a, c^2b)}_*(cK). 
\] 
This follows from an isomorphism of filtered Floer homology 
$\HF^{[a,b)}_*(H) \cong  \HF^{[c^2a, c^2b)}_*(H_c)$, 
where $H_c(x):= c^2 H(x/c)$.

\subsection{Symplectic homology of RCT sets} 

Let us start from our definition of RCT (restricted contact type) sets: 

\begin{defn}\label{defn_RCT_set} 
Let $K$ be a compact subset of $T^*\R^n$. 
\begin{itemize} 
\item $K$ is called a \textit{$C^\infty$-RCT set}, 
if $K$ is connected, $\interior K \ne \emptyset$, $\partial K$ is of $C^\infty$, 
and there exists $X \in \mca{X}(T^*\R^n)$ 
which satisfies  the following properties: 
\begin{itemize} 
\item $L_X \omega_n \equiv \omega_n$, 
\item $X$ points strictly outwards at every point on $\partial K$. 
\end{itemize} 
\item $K$ is called a \textit{RCT set}, 
if there exists a sequence $(K_i)_{i \ge 1}$
which satisfies the following properties: 
\begin{itemize} 
\item $K_i$ is a $C^\infty$-RCT set for every $i$, 
\item $K_{i+1} \subset K_i$ for every $i$, 
\item $\bigcap_{i=1}^\infty K_i = K$. 
\end{itemize} 
\end{itemize} 
\end{defn} 
\begin{rem}
Usually, ``restricted contact type domain'' is defined as a domain (i.e. connected open set) such that its closure is a $C^\infty$-RCT set in the above sense
(see e.g. Definition 1.3 in \cite{Hermann}). 
Thus, the above definition of RCT set is slightly more generalized than the usual definition.
\end{rem} 

$K \subset T^*\R^n$ is called \textit{star-shaped} if there exists $x \in K$ such that 
$ty + (1-t)x \in K$ for any $y \in K$ and $t \in [0,1]$. 
In particular any convex set is star-shaped. 

\begin{lem}\label{starshaped_implies_RCT} 
Any compact and star-shaped set in $T^*\R^n$ is a RCT set. 
\end{lem} 
\begin{proof} 
Suppose that $K \subset T^*\R^n$ is compact and star-shaped. 
We may assume that $(0,\ldots, 0) \in K$ and 
$ty \in K$ for any $t \in [0,1]$ and $y \in K$. 
Let $S:=\{(q,p) \in T^*\R^n \mid |q|^2+|p|^2 = 1\}$. 
Then there exists a function $f: S \to \R_{\ge 0}$ such that 
\[
K = \{ ty \mid y \in S, \, 0 \le t \le f(y) \}.
\]

It is easy to see that $f$ is upper semi-continuous. 
Thus there exists a sequence $(f_j)_{j \ge 1}$ in $C^\infty(S, \R_{>0})$ such that 
$f_j(y) > f_{j+1}(y)$ for every $y \in S$ and $j \ge 1$, and  
$f(y) = \lim_{j \to \infty} f_j(y)$. 
For every $j \ge 1$, 
$K_j:= \{ ty \mid y \in S, \, 0 \le t \le f_j(y) \}$
is a $C^\infty$-RCT set, since $X:= \frac{1}{2} \sum_{i=1}^n p_i \partial_{p_i} + q_i \partial_{q_i}$ 
satisfies $L_X \omega_n = \omega_n$, and is transversal to $\partial K_j$. 
Then $(K_j)_{j \ge 1}$ is a decreasing sequence of $C^\infty$-RCT sets 
satisfying $\bigcap_{j=1}^\infty K_j = K$, 
thus $K$ is a RCT set. 
\end{proof} 

Let $K$ be a $C^\infty$-RCT set in $T^*\R^n$. 
The distribution $\ker (\omega_n|_{\partial K})$ on $\partial K$
defines a $1$-dimensional foliation of $\partial K$, which is called the \textit{characteristic foliation} of $\partial K$. 
\textit{Closed characteristics} are closed leaves of this foliation which are diffeomorphic to $S^1$. 
Let $\mca{P}(\partial K)$ denote the set of closed characteristics. 
The distribution $\ker(\omega_n|_{\partial K})$ is oriented so that
$v \in \ker(\omega_n|_{\partial K})$ is positive if and only if $\omega_n(X, v)>0$, 
where $X$ is any vector on $\partial K$ which points strictly outwards. 
With this orientation, 
for each $\gamma \in \mca{P}(\partial K)$ 
we define its \textit{symplectic action} $\mca{A}(\gamma)$ by 
\[ 
\mca{A}(\gamma):= \int_\gamma \bigg(\sum_i p_i dq_i \bigg). 
\] 

\begin{lem}\label{minimal_period_on_boundary_of_RCT}
Let $K$ be any $C^\infty$-RCT set in $T^*\R^n$. 
Then every $\gamma \in \mca{P}(\partial K)$ satisfies $\mca{A}(\gamma)>0$. 
Moreover, there exists $\gamma_0 \in \mca{P}(\partial K)$ 
such that 
$\mca{A}(\gamma_0) = \inf_{\gamma \in \mca{P}(\partial K)} \mca{A}(\gamma)$. 
\end{lem} 
\begin{proof}
By definition of $C^\infty$-RCT sets, 
there exists $X \in \mca{X}(T^*\R^n)$ 
which satisfies $L_X \omega_n= \omega_n$ 
and points strictly outwards on $\partial K$. 
Let us define $\lambda \in \Omega^1(T^*\R^n)$ by 
$\lambda:= i_X \omega_n$. 
Then $\lambda$ is a contact form on $\partial K$, 
and when $R_\lambda$ denotes its Reeb vector field
(i.e. $i_{R_\lambda}(d\lambda) \equiv 0$ and $\lambda(R_\lambda) \equiv 1$), 
$\mca{P}(\partial K)$ is the set of 
simple closed orbits of $R_\lambda$. 
Moreover, for every $\gamma \in \mca{P}(\partial K)$, 
$\mca{A}(\gamma)$ is equal to the period of $\gamma$ as an orbit of $R_\lambda$. 
Then $\inf_{\gamma \in \mca{P}(\partial K)} \mca{A}(\gamma)$ is positive, 
since $\partial K$ is compact and $R_\lambda$ is nonzero at every point on $\partial K$. 
To show that there exists a closed orbit which attains the infimum, 
let $(\gamma_j)_{j \ge 1}$ be a sequence in $\mca{P}(\partial K)$ 
such that $\mca{A}(\gamma_j)$ converges to the infimum as $j \to \infty$. 
Let us take $p_j$ on $\gamma_j$ for each $j$, 
and let $p$ be the limit of a certain subsequence of $(p_j)_j$. 
Then the orbit $\gamma_0$ which passes through $p$ is closed, 
and $\mca{A}(\gamma_0)$ is equal to the infimum. 
\end{proof} 

For any $C^\infty$-RCT set $K \subset T^*\R^n$, we denote
$c_{\min}(K):= \min_{\gamma \in \mca{P}(\partial K)} \mca{A}(\gamma)$.
When $K$ is convex, 
$c_{\min}(K)$ is also denoted by $c_{\EHZ}(K)$
(see Definiton \ref{defn_EHZ}). 

\begin{lem}\label{SH_of_low_energy}
For any $C^\infty$-RCT set $K \subset T^*\R^n$ 
and $\ep \in (0, c_{\min}(K))$, 
one can assign an isomorphism 
$\SH^{[0,\ep)}_*(K) \cong H_{*+n}(T^*\R^n, T^*\R^n \setminus K)$
so that the diagram 
\[ 
\xymatrix{
H_{*+n}(T^*\R^n, T^*\R^n \setminus K ) \ar[d]_-{\cong} \ar[r] &H_{*+n}(T^*\R^n, T^*\R^n \setminus K') \ar[d]^-{\cong}  \\
\SH^{[0,\ep)}_*(K) \ar[r] & \SH^{[0,{\ep'})}_*(K') 
}
\]
commutes
for any $C^\infty$-RCT sets 
$K' \subset K$
and 
$0< \ep \le \ep' < \min\{ c_{\min}(K), c_{\min}(K')\}$. 
\end{lem}
\begin{proof} 
The isomorphism $\SH^{[0,\ep)}_*(K) \cong H_{*+n}(K, \partial K) \cong H_{*+n}(T^*\R^n, T^*\R^n \setminus K)$ 
follows from  the third bullet in Proposition 4.7 of \cite{Hermann}. 
The commutativity of the diagram follows from the construction of this isomorphism. 
\end{proof} 

\begin{rem}\label{SH_of_low_energy_convex} 
For any convex body $K$ and $\ep \in (0, c_{\EHZ}(K))$, 
there exists a natural isomorphism
$\SH^{[0,\ep)}_*(K) \cong H_{*+n}(T^*\R^n, T^*\R^n \setminus K)$
obtained as 
\[ 
\SH^{[0,\ep)}_*(K) \cong \varinjlim_{K'} \SH^{[0, \ep)}_*(K') \cong \varinjlim_{K'}  H_{*+n}(T^*\R^n, T^*\R^n \setminus K') \cong H_{*+n}(T^*\R^n, T^*\R^n \setminus K), 
\] 
where $K'$ runs over all convex bodies with $C^\infty$ boundaries such that $K' \supset K$. 
The second isomorphism holds since $c_{\EHZ}(K')>\ep$, 
which follows from the monotonicity of the EHZ capacity 
$c_{\EHZ}(K') \ge c_{\EHZ}(K)$. 
\end{rem} 

By Lemma \ref{SH_of_low_energy}, 
for any $C^\infty$-RCT set $K$ we obtain an isomorphism
\[
H_{*+n}(T^*\R^n, T^*\R^n \setminus K) \cong \varprojlim_{\ep \to 0} \SH^{[0,\ep)}_*(K).
\] 
Then, for any $a \in \R_{>0}$, we can define a linear map
\[ 
i^a_K: H_{*+n}(T^*\R^n, T^*\R^n \setminus K)  \cong \varprojlim_{\ep \to 0} \SH^{[0,\ep)}_*(K) \to \SH^{[0,a)}_*(K). 
\] 
The following diagram commutes for any $C^\infty$-RCT sets $K' \subset K$ and $a \le a'$: 
\begin{equation}\label{diagram_K_Kprime_a_aprime}
\xymatrix{
H_{*+n}(T^*\R^n, T^*\R^n \setminus K) \ar[d]_-{i^a_K} \ar[r] &H_{*+n}(T^*\R^n, T^*\R^n \setminus K')  \ar[d]^-{i^{a'}_{K'}}  \\
\SH^{[0,a)}_*(K) \ar[r] & \SH^{[0,a')}_*(K'). 
}
\end{equation}
Also, the following diagram commutes for any $c \in \R_{>0}$: 
\begin{equation}\label{diagram_K_a_c} 
\xymatrix{
H_{*+n}(T^*\R^n, T^*\R^n \setminus K) \ar[d]_-{i^a_K} \ar[r]^-{\cong}&H_{*+n}(T^*\R^n, T^*\R^n \setminus cK) \ar[d]^-{i^{c^2a}_{cK}}  \\
\SH^{[0,a)}_*(K) \ar[r]_-{\cong} & \SH^{[0,c^2a)}_*(cK). 
}
\end{equation}

Now let us define the map 
$i^a_K: H_{*+n}(T^*\R^n, T^*\R^n \setminus K) \to \SH^{[0,a)}_*(K)$ 
for any RCT set $K$ and $a \in \R_{>0}$. 
Notice that there are natural isomorphisms 
\begin{align*} 
H_{*+n}(T^*\R^n, T^*\R^n \setminus K) &\cong \varinjlim_{K'} H_{*+n}(T^*\R^n, T^*\R^n \setminus K'), \\ 
\SH^{[0,a)}_*(K) &\cong \varinjlim_{K'} \SH^{[0,a)}_*(K'), 
\end{align*} 
where $K'$ runs over all $C^\infty$-RCT sets with $K' \supset K$. 
Then one can define $i^a_K$ 
as the limit of $(i^a_{K'})_{K' \supset K}$. 

\subsection{Symplectic homology capacity $c_{\SH}$}

To define the capacity $c_{\SH}$, we first need the following definition. 
Recall that, in this paper all (co)homology groups are defined over $\Z/2$, unless otherwise specified. 

\begin{defn} 
For any $\R$-vector space $V$ of dimension $d \in \Z_{>0}$ and a compact subset $K \subset V$, 
we define $\nu^V_K \in H_d(V, V \setminus K)$ in the following manner. 
\begin{itemize} 
\item If $K$ is convex, then $H_d(V, V \setminus K) \cong \Z/2$. 
Then we define $\nu^V_K$ to be the unique non-zero element of $H_d(V, V \setminus K)$. 
\item When $K$ is an arbitrary compact subset of $V$, 
take a compact convex set $K' \subset V$ satisfying $K \subset K'$, 
and let $i_{KK'}: H_d(V, V \setminus K') \to H_d(V, V \setminus K)$
be the linear map induced by $\id_V: (V, V \setminus K') \to (V, V \setminus K)$. 
Then it is easy to see that $i_{KK'}(\nu^V_{K'})$ does not depend on the choice of $K'$. 
Then we define $\nu^V_K:= i_{KK'}(\nu^V_{K'})$. 
\end{itemize}
\end{defn} 

Now, for any RCT set $K \subset T^*\R^n$, we define 
\[ 
c_{\SH}(K):= \inf \{ a \in \R_{>0} \mid i^a_K(\nu^{T^*\R^n}_K) = 0\}. 
\] 
The invariant $c_{\SH}$ will be called \textit{symplectic homology capacity}. 
The next lemma summarizes some properties of the capacity $c_{\SH}$. 
The properties (i), (ii), (iii) are (respectively) called conformality, monotonicity, and spectrality. 

\begin{lem}\label{properties_of_c_SH} 
\begin{enumerate}
\item[(i):] For any RCT set $K$ and $c \in \R_{>0}$, there holds $c_{\SH}(cK)= c^2 c_{\SH}(K)$. 
\item[(ii):] For any RCT sets $K' \subset K$, there holds $c_{\SH}(K') \le c_{\SH}(K)$. 
\item[(iii):] For any $C^\infty$-RCT set $K$, 
there exist $\gamma \in \mca{P}(\partial K)$ 
and $m \in \Z_{\ge 1}$ such that 
$c_{\SH}(K) =  m \cdot \mca{A}(\gamma)$. 
In particular $c_\SH(K) \ge c_{\min}(K)$. 
\end{enumerate}
\end{lem} 
\begin{proof}
(i) follows from the commutativity of (\ref{diagram_K_a_c}), 
and 
(ii) follows from the commutativity of (\ref{diagram_K_Kprime_a_aprime}). 
(iii) is proved in Corollary 5.8 of \cite{Hermann} 
under the assumption that $\partial K$ has a nice action spectrum 
(see pp. 342 of \cite{Hermann} for its definition). 
Since $\partial K$ has a nice action spectrum for $C^\infty$-generic $K$ (Proposition 2.5 of \cite{Hermann}), 
one can remove this assumption by the limiting argument. 
\end{proof}

\subsection{$S^1$-equivariant symplectic homology capacities}

For any $C^\infty$-RCT set $K \subset T^*\R^n$ 
(in general, for any Liouville domain) 
and $a \in \R_{>0}$, 
one can define the $S^1$-equivariant symplectic homology 
$\SH^{[0,a), S^1}_*(K)$ 
and a linear map 
\[ 
(i^a_K)^{S^1}: H^{S^1}_{*+n}(T^*\R^n, T^*\R^n \setminus K)  \to \SH^{[0, a),S^1}_*(K), 
\] 
where $H^{S^1}_*(T^*\R^n, T^*\R^n \setminus K)$ 
is the $S^1$-equivariant homology with the trivial $S^1$-action on $(T^*\R^n, T^*\R^n \setminus K)$, 
thus canonically isomorphic to $H_*(T^*\R^n, T^*\R^n \setminus K)\otimes H_*(\C P^\infty)$. 
For each $k \in \Z_{\ge 1}$, let 
\[ 
c^k_{\SH^{S^1}}(K):= \inf \{ a \mid (i^a_K)^{S^1} (\nu^{T^*\R^n}_K \otimes [\C P^{k-1}]) = 0 \}. 
\] 
Let us call the invariants $c^k_{\SH^{S^1}}\,(k \ge 1)$ \textit{equivariant symplectic homology capacities}. 
\begin{rem} 
This construction goes back at least to Section 5.3 of Viterbo \cite{Viterbo_GAFA}, 
where the Floer-theoretic analogue of the equivariant Ekeland-Hofer capacities \cite{Ekeland_Hofer} 
was introduced. 
This construction is revisited in recent papers such as 
Gutt-Hutchings \cite{Gutt_Hutchings} and 
Ginzburg-Shon \cite{Ginzburg_Shon}. 
In particular, \cite{Gutt_Hutchings} introduced 
a sequence of capacities using positive equivariant symplectic homology with rational coefficients, 
established basic properties of these capacities, 
and gave combinatorial formulas to compute these capacities of convex and concave toric domains. 
In \cite{Gutt_Hutchings} it is conjectured that 
the Gutt-Hutchings capacities are equal to the equivariant Ekeland-Hofer capacities for any compact star-shaped domain 
(Conjecture 1.9 of \cite{Gutt_Hutchings}). 
\end{rem} 

For any $C^\infty$-RCT set $K$, there holds the following inequalities: 
\begin{equation}\label{min_SHS1_SH} 
c_{\min}(K)\le c^1_{\SH^{S^1}}(K)\le c_\SH(K). 
\end{equation} 
For the first inequality, see the ``contractible Reeb orbits'' property in Theorem 1.24 of \cite{Gutt_Hutchings}. 
For the second inequality, see Lemma 3.2 of \cite{Ginzburg_Shon}. 

\begin{rem} 
One has to be careful since \cite{Gutt_Hutchings} and \cite{Ginzburg_Shon} use $\Q\,$-coefficients, 
while we work over $\Z/2\,$-coefficients. 
Also, the definitions of equivariant capacities in these papers use positive (equivariant) symplectic homology, 
and are superficially different from our definition. 
However, it is straightforward to see that the proofs in these papers also work in our setting. 
\end{rem} 

F. Schlenk \cite{Schlenk} pointed out that 
Theorem \ref{SH=EHZ}, combined with (\ref{min_SHS1_SH}), 
implies the following corollary: 

\begin{cor} 
$c_{\EHZ}(K) = c^1_{\SH^{S^1}}(K) = c_{\SH}(K)$
for any convex body $K$ in $T^*\R^n$. 
\end{cor} 

\section{Symplectic homology and loop space homology} 

Let $\pr: T^*\R^n \to \R^n$ denote the natural projection map, namely 
$\pr(q,p):=q$.
For any $q \in \R^n$, 
we identify $T_q^*\R^n$ with $\pr^{-1}(q)$. 

\begin{defn} 
$K \subset T^*\R^n$ is called \textit{fiberwise convex} 
if $K_q:= K \cap T_q^*\R^n$ is a convex set in $T_q^*\R^n$ 
for every $q \in \R^n$. 
\end{defn} 

Throughout this section, $K$ denotes a nonempty, compact and fiberwise convex set in $T^*\R^n$. 
In Section 3.1, we state 
Theorem \ref{SH_and_loop_space_homology}, 
which shows that symplectic homology of $K$ is isomorphic to a certain relative homology of loop spaces of $\R^n$. 
The proof of Theorem \ref{SH_and_loop_space_homology} is carried out in Section 4. 
In Section 3.2, we deduce 
Corollary \ref{c_SH_and_loop_space_homology} from Theorem \ref{SH_and_loop_space_homology}, 
which shows that the capacity $c_{\SH}(K)$ is equal to a certain min-max value defined from homology of loop spaces. 
In Section 3.3, we prove some technical results about fiberwise convex functions, 
which are used in Section 3.1 and in the proof of Theorem \ref{SH_and_loop_space_homology} (see Section 4.6). 

\subsection{Symplectic homology and loop space homology} 

Let $\Lambda$ denote the space of $L^{1,2}$-maps from $S^1=\R/\Z$ to $\R^n$, 
equipped with the $L^{1,2}$-topology. 
For each $\gamma \in \Lambda$, we define 
$\len_K(\gamma)$ as follows: 
\begin{equation}\label{190420_1} 
\len_K(\gamma):= \begin{cases} 
\int_{S^1} \, (\max_{p \in K_{\gamma(t)}} \, p \cdot \dot{\gamma}(t) \,) \, dt \, &(\gamma(S^1) \subset \pr(K)) \\
-\infty &(\gamma(S^1) \not\subset \pr(K)). 
\end{cases}
\end{equation}

\begin{ex}
If $K$ is the unit disk cotangent bundle of $\pr(K)$, namely 
\[ 
K = \{ (q,p) \in T^*\R^n \mid q \in \pr(K), \, |p| \le 1\}, 
\] 
then $\len_K(\gamma) = \int_{S^1} |\dot{\gamma}(t)| \, dt$
for any $\gamma \in \Lambda$ satisfying 
$\gamma(S^1) \subset \pr(K)$. 
\end{ex}

Let us summarize elementary properties of $\len_K$. 

\begin{lem}\label{properties_of_len_K}
Let $K$ be any nonempty, compact, and fiberwise convex set in $T^*\R^n$. 
\begin{enumerate}
\item[(i):]  (\ref{190420_1}) is well-defined. Namely, for any $\gamma \in \Lambda$ satisfying $\gamma(S^1) \subset \pr(K)$, the function 
$\rho_\gamma: S^1 \to \R;\,t \mapsto \max_{p \in K_{\gamma(t)}}  p \cdot \dot{\gamma}(t)$ is integrable. 
\item[(ii):] $\len_K$ is upper semi-continuous. 
Namely, 
if a sequence $(\gamma_k)_k$ in $\Lambda$ converges to $\gamma \in \Lambda$ in the $L^{1,2}$-topology, 
then $\len_K(\gamma)\ge \limsup_k \len_K(\gamma_k)$.
\item[(iii):] 
Suppose that $\partial K$ is of $C^\infty$ and strictly convex. 
Let 
$\gamma: S^1 \to \interior(\pr(K))$ 
be a $C^\infty$-map such that 
$\dot{\gamma}(t) \ne 0$ for every $t \in S^1$. 
Then, for every $t \in S^1$
there exists unique $p_\gamma(t) \in K_{\gamma(t)}$ such that 
$p_\gamma(t) \cdot \dot{\gamma}(t) = \max _{p \in K_{\gamma(t)}}  p \cdot \dot{\gamma}(t)$. 
Moreover, $\bar{\gamma}: S^1 \to \partial K$ defined by 
$\bar{\gamma}(t):= (\gamma(t), p_\gamma(t))$ is of $C^\infty$, 
and satisfies 
\[ 
\len_K (\gamma) = \int_{S^1} \bar{\gamma}^*\bigg(\sum_{i=1}^n p_i dq_i \bigg). 
\] 
\item[(iv):] 
Suppose that $\partial K$ is of $C^\infty$ and strictly convex. 
Then $\len_K$ is continuous on 
$\{ \gamma \in \Lambda \mid \gamma(S^1)\subset \pr(K)\}$
with respect to the $L^{1,2}$-topology.
\item[(v):] 
Let $K'$ be any nonempty, compact, and fiberwise convex set in $T^*\R^n$
which satisfies $K' \subset K$. 
Then $\len_{K'}(\gamma) \le \len_K(\gamma)$ for any $\gamma \in \Lambda$. 
\end{enumerate} 
\end{lem} 
\begin{proof} 
(i) and (ii) are consequences of Lemmas 
\ref{sequence_of_fiberwise_convex_functions} and \ref{limit_of_Legendrean_duals}. 
Let us take a sequence $(H_j)_{j \ge 1}$ as in Lemma \ref{sequence_of_fiberwise_convex_functions}, 
and let $L_{H_j}$ denote the Legendre dual of $H_j$ (see Lemma \ref{limit_of_Legendrean_duals} for the definition of Legendre dual). 

Let us prove (i). 
Since $K$ is compact, there exists $C>0$ such that 
$|p| \le C$ for every $(q,p) \in K$. 
Then $|\rho_\gamma|  \le C \cdot |\dot{\gamma}|$ 
for every $\gamma \in \Lambda$ satisfying $\gamma(S^1) \subset \pr(K)$. 
Since $|\dot{\gamma}|$ is integrable, 
it is sufficient to show that $\rho_\gamma$ is measurable. 
Lemma \ref{limit_of_Legendrean_duals} (ii) 
says that $\rho_\gamma(t) = \lim_{j \to \infty} L_{H_j}(\gamma(t), \dot{\gamma}(t))$ for every $t \in S^1$. 
Then $\rho_\gamma$ is measurable, since $L_{H_j}(\gamma, \dot{\gamma})$ is obviously measurable for every $j$. 

Let us prove (ii). 
For each $j$, let us define $\mca{L}_j: \Lambda \to \R$ by 
$\mca{L}_j(\gamma):= \int_{S^1}  L_{H_j}(\gamma, \dot{\gamma}) \, dt$. 
Then $(\mca{L}_j)_{j \ge 1}$ is a decreasing sequence of continuous functions on $\Lambda$, 
and $\len_K = \lim_{j \to \infty} \mca{L}_j$ by Lemma \ref{limit_of_Legendrean_duals}. 
Then $\len_K$ is upper semi-continuous. 

Let us prove (iii). 
Since $\partial K$ is of $C^\infty$ and strictly convex, 
$\partial K_q$ is of $C^\infty$ and strictly convex
for any $q \in \interior (\pr(K))$. 
Then, for any $t \in S^1$, 
there exists unique $p_\gamma(t) \in K_{\gamma(t)}$ which satisfies
$\max_{p \in K_{\gamma(t)}} p \cdot \dot{\gamma}(t) = p_\gamma(t) \cdot \dot{\gamma}(t)$. 
Moreover, $\bar{\gamma}=(\gamma, p_\gamma)$ is of $C^{\infty}$ by the inverse mapping theorem. 
The last assertion follows from 
$\bar{\gamma}^*\bigg( \sum_i p_i dq_i \bigg) = p_\gamma(t) \cdot \dot{\gamma}(t) \, dt$, 
which is straightforward. 

Let us prove (iv). 
First we prove that 
\[ 
c:\pr(K) \times\R^n\to\R; \quad (q,v) \mapsto  \max_{p \in K_q} \, p \cdot v
\] 
is continuous. 
Let $(q_k, v_k)_{k \ge 1}$ be a sequence on $\pr(K) \times \R^n$ 
which converges to $(q_\infty, v_\infty)$ as $k \to \infty$. 
Then we want to show 
$\lim_{k \to \infty} \max_{p \in K_{q_k}} p \cdot v_k = \max_{p \in K_{q_\infty}} p \cdot v_\infty$. 
By the compactness of $K$ one has 
$\limsup_{k \to \infty} \max_{p \in K_{q_k}} p \cdot v_k \le \max_{p \in K_{q_\infty}} p \cdot v_\infty$, 
thus it is sufficient to show 
$\liminf_{k \to \infty} \max_{p \in K_{q_k}} p \cdot v_k \ge  \max_{p \in K_{q_\infty}} p \cdot v_\infty$. 
Take $p_\infty \in K_{q_\infty}$ so that $p_\infty \cdot v_\infty = \max_{p \in K_{q_\infty}} p \cdot v_\infty$. 
We claim that there exists a sequence $(p_k)_k$ such that $(q_k, p_k) \in K$ for every $k$ and $\lim_{k \to \infty} p_k= p_\infty$. 
This claim can be verified as follows: 
\begin{itemize} 
\item If $q_\infty \in \interior(\pr(K))$, then 
there exists $\ep>0$ such that the closed $\ep$-neighborhood of $q_\infty$ is contained in $\pr(K)$. 
For each $k \ge 1$ such that $|q_k-q_\infty|<\ep$, let us define $p_k$ as follows: 
\begin{itemize}
\item If $q_k=q_\infty$, then $p_k:=p_\infty$.
\item If $q_k \ne q_\infty$, then there exist $(q'_k, p'_k) \in K$ and $t_k \in (0,1)$ such that 
$|q_\infty - q'_k|=\ep$ and $q_k = t_k q'_k + (1-t_k )q_\infty$. 
Then $p_k:= t_k  p'_k + (1-t_k )p_\infty$. 
\end{itemize} 
Then it is easy to see that $\lim_{k \to \infty} p_k = p_\infty$. 
\item 
If $q_\infty \in \partial(\pr(K))$, then $K_{q_\infty}= \{ p_\infty\}$ since $\partial K$ is strictly convex, 
thus any sequence $(p_k)_k$ satisfying $(q_k, p_k) \in K \,(\forall k)$ satisfies 
$\lim_{k \to \infty} p_k = p_\infty$. 
\end{itemize} 
Now we can finish the proof of the continuity of $c$ by 
$\liminf_{k \to \infty} \max_{p \in K_{q_k}} p \cdot v_k \ge  \lim_{k \to \infty} p_k \cdot v_k = p_\infty \cdot v_\infty$. 

Now suppose that (iv) does not hold. 
Then there exists a sequence $(\gamma_k)_k$ 
in $\{ \gamma \in \Lambda \mid \gamma(S^1) \subset \pr(K)\}$ 
which converges to $\gamma_\infty$ in the $L^{1,2}$-topology, 
and $\inf_k  |\len_K(\gamma_k) - \len_K(\gamma_\infty)|>0$. 
By replacing $(\gamma_k)_k$ with its subsequence if necessary, 
we may assume $\lim_{k \to \infty} \dot{\gamma}_k (t) = \dot{\gamma}_\infty(t)$ for almost every $t \in S^1$. 
On the other hand $\lim_{k \to \infty} \gamma_k(t) = \gamma_\infty(t)$ for every $t$. 
Thus $\lim_{k \to \infty} c(\gamma_k(t), \dot{\gamma}_k(t))  = c(\gamma_\infty(t), \dot{\gamma}_\infty(t))$ for almost every $t$, 
which implies $\lim_{k \to \infty} \len_K(\gamma_k) = \len_K(\gamma_\infty)$, 
contradicting our assumption. 

Finally, (v) follows from $\pr(K') \subset \pr(K)$ and 
$\max_{p \in K'_q} \, (p \cdot v) \le \max_{p \in K_q} \, (p \cdot v)$
for any $q \in \pr(K')$ and $v \in \R^n$. 
\end{proof} 

For any $a \in \R$, let 
$\Lambda^a_K:= \{ \gamma \in \Lambda \mid \len_K(\gamma) < a\}$. 
By Lemma \ref{properties_of_len_K} (ii), this is open in $\Lambda$ with the $L^{1,2}$-topology. 
Moreover, Lemma \ref{properties_of_len_K} (v) shows that 
if $K' \subset K$ then $\Lambda^a_{K} \subset \Lambda^a_{K'}$. 

\begin{thm}\label{SH_and_loop_space_homology} 
For any nonempty, compact and fiberwise convex set $K \subset T^*\R^n$ 
and real numbers $a<b$, 
one can assign an isomorphism 
\[ 
\SH^{[a,b)}_*(K) \cong H_*(\Lambda^b_K, \Lambda^a_K)
\] 
so that the diagram 
\begin{equation}\label{SH_aba'b'}
\xymatrix{
\SH^{[a,b)}_*(K) \ar[r]^-{\cong}\ar[d]&H_*(\Lambda^b_K, \Lambda^a_K) \ar[d]\\
\SH^{[a',b')}_*(K') \ar[r]_-{\cong} &H_*(\Lambda^{b'}_{K'}, \Lambda^{a'}_{K'})
}
\end{equation}
commutes for any $a \le a'$, $b \le b'$ and any fiberwise convex $K' \subset K$. 
\end{thm} 

\begin{rem} 
If the boundary of $\pr(K) \subset \R^n$ is of $C^\infty$ and 
$K$ is the unit disk cotangent bundle of $\pr(K)$, then 
Theorem \ref{SH_and_loop_space_homology} is essentially equivalent to Theorem 1.1 of \cite{JSG}. 
\end{rem} 

\begin{rem} 
It is likely that Theorem \ref{SH_and_loop_space_homology} naturally extends to 
any nonempty, compact and fiberwise convex set $K \subset T^*Q$ where $Q$ is an arbitrary closed manifold. 
However, since our main applications (Theorem \ref{SH=EHZ} and Theorem \ref{HZ_subadditivity}) make sense only on symplectic vector spaces, 
in this paper we work on symplectic vector spaces. 
\end{rem}

\subsection{Symplectic homology capacity and loop space homology}

In this subsection, we prove a formula (Corollary \ref{c_SH_and_loop_space_homology})
which computes $c_{\SH}(K)$ in terms of homology of loop spaces of $\R^n$. 
Let us recall from Section 2.4 that for any RCT set $K$, 
\[ 
c_{\SH}(K) = \inf\{a \in \R_{>0}  \mid i^a_K(\nu^{T^*\R^n}_K)=0\}. 
\] 

For any $a \in \R_{>0}$, let us consider a map 
\[ 
j^a_K: (\R^n, \R^n \setminus \pr(K) ) \to (\Lambda^a_K, \Lambda^0_K)
\] 
which sends each $q \in \R^n$ to the constant loop at $q$. 

\begin{lem}\label{i_a_K_and_j_a_K}
Let $K$ be any RCT set in $T^*\R^n$ which is fiberwise convex.
Then, for any $a \in \R_{>0}$, 
$i^a_K(\nu^{T^*\R^n}_K)\in \SH^{[0,a)}_n(K)$ 
corresponds to 
$H_*(j^a_K)(\nu^{\R^n}_{\pr(K)}) \in H_n(\Lambda^a_K, \Lambda^0_K)$ 
via the isomorphism 
$\SH^{[0,a)}_*(K) \cong H_*(\Lambda^a_K, \Lambda^0_K)$. 
\end{lem} 
\begin{proof} 
For any $R \in \R_{>0}$ let 
$K_R:= \{ (q,p) \in T^*\R^n \mid |q|, |p|\le R\}$. 

First notice that it is sufficient to prove the lemma for $K = K_R$ for every $R$. 
Indeed, for any compact $K \subset T^*\R^n$, there exists $R$ such that $K \subset K_R$. 
By the commutativity of (\ref{SH_aba'b'}), we have a commutative diagram 
\[ 
\xymatrix{
\SH^{[0,a)}_*(K_R) \ar[r]\ar[d]_-{\cong}& \SH^{[0,a)}_*(K) \ar[d]^-{\cong} \\
H_*(\Lambda^a_{K_R}, \Lambda^0_{K_R}) \ar[r] & H_*(\Lambda^a_K, \Lambda^0_K). 
}
\] 
Then the upper horizontal map sends 
$i^a_{K_R}(\nu^{T^*\R^n}_{K_R})$ to
$i^a_K(\nu^{T^*\R^n}_K)$. 
Assuming that we have proved the lemma for $K_R$, the left vertical map sends 
$i^a_{K_R}(\nu^{T^*\R^n}_{K_R})$ to 
$H_*(j^a_{K_R})(\nu^{\R^n}_{\pr(K_R)})$, 
which is sent to 
$H_*(j^a_K)(\nu^{\R^n}_{\pr(K)})$ 
by the lower horizontal map. 
By the commutativity of the diagram, the right vertical map sends 
$i^a_K(\nu^{T^*\R^n}_K)$
to 
$H_*(j^a_K)(\nu^{\R^n}_{\pr(K)})$, 
which completes the proof for $K$. 

Thus it is sufficient to consider the case $K=K_R$. 
It is also sufficient to consider the case when $a$ is sufficiently small, 
since for any $a<b$ we have a commutative digram 
 \[ 
\xymatrix{
\SH^{[0,a)}_*(K_R) \ar[r]\ar[d]_-{\cong}& \SH^{[0,b)}_*(K_R) \ar[d]^-{\cong} \\
H_*(\Lambda^a_{K_R}, \Lambda^0_{K_R}) \ar[r] & H_*(\Lambda^b_{K_R}, \Lambda^0_{K_R}). 
}
\] 

Moreover, it is sufficient to prove 
$H_*(j^a_{K_R})(\nu^{\R^n}_{\pr(K_R)}) \ne 0$ for sufficiently small $a$. 
Indeed, when $a$ is sufficiently small, 
Remark \ref{SH_of_low_energy_convex} implies that 
$\SH^{[0,a)}_n(K_R) \cong \Z/2$ is generated by $i^a_{K_R}(\nu^{T^*\R^n}_{K_R})$. 
Then the isomorphism $\SH^{[0,a)}_*(K_R) \cong H_n(\Lambda^a_{K_R}, \Lambda^0_{K_R})$ 
maps 
$i^a_{K_R}(\nu^{T^*\R^n}_{K_R})$ 
to 
the only nonzero element in $H_n(\Lambda^a_{K_R}, \Lambda^0_{K_R})$, 
that is $H_*(j^a_{K_R})(\nu^{\R^n}_{\pr(K_R)})$. 

The rest of the proof 
is essentially the same as 
the proof of Lemma 6.6 (2) of \cite{JSG}, which we repeat here for 
the sake of completeness.
For any $\gamma \in \Lambda$ let $\len(\gamma):= \int_{S^1} |\dot{\gamma}(t)| \, dt$, 
and for any $a \in \R_{>0}$ let 
$U^a:= \{ \gamma \in \Lambda \mid \len(\gamma) < a/R\}$. 
Also let $B_R:= \{ q \in \R^n \mid |q| \le R\}$
and $V_R:= \{ \gamma \in \Lambda \mid \gamma(S^1)\not\subset B_R\}$. 
Then
\[
\Lambda^a_{K_R}=U^a\cup V_R,\quad \Lambda^0_{K_R}=V_R.
\]
Since both $U^a$ and $V_R$ are open sets in $\Lambda$, the inclusion map 
\[ 
(U^a, U^a\cap V_R) \to (U^a \cup V_R, V_R) = (\Lambda^a_{K_R}, \Lambda^0_{K_R})
\]
induces an isomorphism on homology. 
Thus it is sufficient to show that 
\[
c^a_R: (\R^n, \R^n\setminus B_R) \to (U^a, U^a\cap V_R)
\] 
which sends each $q \in \R^n$ to the constant loop at $q$, 
induces an injection on homology if $a$ is sufficiently small. 

Let us define $\ev:\Lambda\to\R^n$ by $\ev(\gamma):=\gamma(0)$. 
If $a$ is sufficiently small, then $\ev$ maps $U^a \cap V_R$ to $\R^n\setminus\{0\}$, 
and we obtain a commutative diagram
\[ 
\xymatrix{
(\R^n, \R^n\setminus B_R) \ar[r]^-{c^a_R}\ar[rd]_{\id_{\R^n}}&(U^a, U^a \cap V_R)\ar[d]^-{\ev}\\
&(\R^n, \R^n\setminus\{0\}).\\
}
\] 
The diagonal map induces an isomorphism on homology, thus $H_*(c^a_R)$ is injective. 
This completes the proof. 
\end{proof} 

As an immediate corollary of Lemma \ref{i_a_K_and_j_a_K}, 
we obtain the following formula which computes 
$c_{\SH}(K)$ from homology of loop spaces. 

\begin{cor}\label{c_SH_and_loop_space_homology} 
For any RCT set $K \subset T^*\R^n$ which is fiberwise convex, 
\[
c_{\SH}(K) = \inf \{a \in \R_{>0} \mid H_*(j^a_K)(\nu^{\R^n}_{\pr(K)}) =0\}.
\]
\end{cor} 

\subsection{Technical results on fiberwise convex functions} 

In this subsection we prove some preliminary results on (fiberwise) convex functions. 

\begin{defn}
For any (finite-dimensional) real vector space $V$, $f \in C^0(V, \R)$ is called \textit{convex} if 
$f(tx+(1-t)y) \le tf(x) + (1-t) f(y)$ for any $x,y \in V$ and $t \in [0, 1]$. 
$f \in C^2(V, \R)$ is called \textit{strictly convex}, 
if for any $x \in V$, the Hessian of $f$ at $x$ (which is a symmetric bilinear form on $V$) is positive definite. 
$f \in C^0(T^*\R^n)$ is called \textit{fiberwise convex} if $f|_{T^*_q\R^n}$ is convex for every $q \in \R^n$, 
and $f \in C^2(T^*\R^n)$ is called \textit{fiberwise strictly convex} if $f|_{T^*_q\R^n}$ is strictly convex for every $q \in \R^n$. 
\end{defn} 

For any $a \in \R_{>0}$, let us define $Q_a \in C^\infty(T^*\R^n)$ by 
$Q_a(q,p):= a(|q|^2 + |p|^2)$. 

\begin{lem}\label{sequence_of_fiberwise_convex_functions} 
For any nonempty, compact, and fiberwise convex set $K \subset T^*\R^n$, 
there exist sequences 
$(a_j)_{j \ge 1}$ and 
$(H_j)_{j \ge 1}$ which satisfy the following properties: 
\begin{enumerate} 
\item[(i):]  $(a_j)_j$ is a strictly increasing sequence in $\R_{>0}\setminus \pi\Z$.
\item[(ii):] $\lim_{j \to \infty} a_j  = \infty$. 
\item[(iii):] $(H_j)_j$ is a strictly increasing sequence of fiberwise strictly convex $C^\infty$-functions on $T^* \R^n$. 
\item[(iv):] For every $j$, there exists $b_j \in \R$ such that $H_j$ is a compact perturbation of $Q_{a_j} + b_j$, 
i.e. $H_j - (Q_{a_j}+b_j)$ is compactly supported. 
\item[(v):] $\lim_{j \to \infty} H_j(q,p)=\begin{cases}\infty&((q,p)\notin K)\\0&((q,p)\in K).\end{cases}$
\end{enumerate} 
\end{lem} 
\begin{proof} 
Let us take a sequence $(U_j)_j$ of open sets in $T^*\R^n$ such that 
$\overline{U_{j+1}} \subset U_j$ for every $j$, 
and $\bigcap_{j=1}^\infty  U_j = K$. 

Let us consider conditions (ii') and (v') as follows: 
\begin{enumerate} 
\item[(ii'):]   $a_j > 2^j$ for every $j$. 
\item[(v'):] 
The following properties hold for every $j$: 
\begin{itemize} 
\item $H_j(q,p)  > 2^j$ if $(q,p) \notin U_j$, 
\item $-\frac{1}{2^j} < H_j(q,p) < - \frac{1}{2^{j+1}}$ if $(q,p) \in K$. 
\end{itemize} 
\end{enumerate} 
Obviously (ii') implies (ii), and (v') implies (v). 
Thus it is sufficient to construct sequences 
$(a_j)_j$ and $(H_j)_j$ 
satisfying (i), (ii'), (iii), (iv), (v'). 
We are going to construct such sequences by induction on $j$. 
Suppose that we have defined $a_1, \ldots, a_{j-1}$ and $H_1, \ldots, H_{j-1}$ 
satisfying these conditions. 
In the following argument 
we construct a pair $(a_j, H_j)$ so that these conditions are satisfied. 
Let us take $a \in \R_{>0} \setminus \pi \Z$ such that $a > \max\{a_{j-1}, 2^j \}$. 
We fix such $a$ in the rest of the proof. 

\textbf{Step 1.} 
For any $b\in \R_{\ge 0}$, 
we define $F_b: T^*\R^n \to \R$ in the following way. 

For each $q \in \R^n$, 
let $\mca{F}(b,q)$ denote the set of convex functions 
$f: T_q^*\R^n \to \R$ 
satisfying the following conditions: 
\begin{itemize} 
\item $f(p) \le Q_a(q,p) + b$ for every  $p \in T_q^*\R^n$. 
\item $f(p) \le - \frac{3}{2^{j+2}}$ if $(q,p) \in K$. 
\end{itemize} 
Let us define 
$F_b$ by 
$F_b(q,p):= \sup_{f \in \mca{F}(b,q)}  f(p)$. 
Then, $F_b|_{T_q^*\R^n}$ is convex (thus continuous) for every $q \in \R^n$. 
The function $F_b$ satisfies the following properties: 
\begin{itemize} 
\item[(1-0):] If $q \notin \pr(K)$ then $F_b(q,p) = Q_a(q,p)+b$. 
\item[(1-1):]$F_b$ is a compact perturbation of $Q_a+b$. 
\item[(1-2):]$F_b(q,p) \ge - \frac{3}{2^{j+2}}$ for every $(q,p) \in T^*\R^n$. 
\item[(1-3):] $F_b(q,p) = - \frac{3}{2^{j+2}}$ if $(q,p) \in K$. 
\item[(1-4):] For any $\ep>0$, there exists $\delta>0$ such that 
if $p \in T_q^*\R^n$ satisfies $\dist(K_q, p)<\delta$, 
where $\dist$ denotes the Euclidean distance on $T_q^*\R^n$, 
then $F_b(q,p) < -\frac{3}{2^{j+2}} + \ep$.
\end{itemize} 
(1-0) holds since $Q_a(q,p)+b \in \mca{F}(b, q)$ if $q \notin \pr(K)$. 
(1-2) and (1-3) hold since the constant function $-\frac{3}{2^{j+2}}$ is an element of $\mca{F}(b, q)$. 
(1-1) holds since if $|q|^2+|p|^2$ is sufficiently large, the linear function 
\[ 
T_q^*\R^n \to \R; \, x \mapsto Q_a(q,p) +  b +  2a (p \cdot (x-p))
\] 
is an element of $\mca{F}(b,q)$. 
(1-4) follows from (1-3), (1-1) and the convexity of $F_b|_{T_q^*\R^n}$. 

Moreover, when $b$ is sufficiently large, the following properties hold: 
\begin{itemize} 
\item[(1-5):] $F_b(q,p)>H_{j-1}(q,p)$ for any $(q,p)\in T^*\R^n$. 
\item[(1-6):] $F_b(q,p) > 2^j$ if $(q,p) \notin U_j$. 
\end{itemize} 

Let us check that (1-5) holds for sufficiently large $b$. 
By the induction assumption, 
$H_{j-1}<-\frac{1}{2^j}$ on $K$. 
Thus $H_{j-1} + \frac{1}{2^{j+2}} < -\frac{3}{2^{j+2}}$ on $K$. 
Since $H_{j-1}$ is a compact perturbation of $Q_{a_{j-1}} + b_{j-1}$ and $a_{j-1}<a$, 
when $b$ is sufficiently large 
$H_{j-1} + \frac{1}{2^{j+2}}  \in \mca{F}(b,q)$. 
This means that $H_{j-1}(q,p) +\frac{1}{2^{j+2}} \le  F_b(q,p)$ for any $(q,p) \in T^*\R^n$, 
thus (1-5) holds. 

Let us check that (1-6) holds for sufficiently large $b$. 
For any $(q,p) \notin U_j$ such that $K_q \ne \emptyset$, 
let $p'$ be the unique point on $K_q$ such that $|p-p'| = \dist(K_q, p)$. 

\begin{rem} 
The uniqueness of $p'$ follows from the convexity of $K_q$. 
Indeed, suppose that there exist $p' \ne p''$ in $K_q$ satisfying 
$|p-p'| = |p-p''| = \dist(K_q, p)$. 
Then $p''':= (p'+p'')/2 \in K_q$ by the convexity of $K_q$. 
On the other hand $|p-p'''|<\dist(K_q, p)$ by $p' \ne p''$, 
which contradicts $p''' \in K_q$. 
\end{rem} 

Let us define a linear function $H_{q,p}$ on $T_q^*\R^n$ by 
\[ 
H_{q,p}(x):= -\frac{3}{2^{j+2}}  + (x-p') \cdot (p-p') \cdot \frac{2^j + 1 + 3/2^{j+2} }{|p-p'|^2}. 
\] 
Then $H_{q,p}(p)=2^j+1$ and $H_{q,p} \le -\frac{3}{2^{j+2}}$ on $K_q$. 
Also, there holds 
\[ 
S:= \sup_{\substack{(q,p) \notin U_j \\ K_q \ne \emptyset}} 
(\max_{x \in T_q^*\R^n}   H_{q,p}(x) - Q_a(q,x) ) < \infty. 
\] 
This can be checked as follows: for any $(q,p) \notin U_j $ with $K_q \ne \emptyset$, 
\[ 
\max_{x \in T_q^*\R^n}  H_{q,p}(x) - Q_a(q,x)  = H_{q,p}(0) + \frac{|\nabla H_{q,p}|^2}{4a}  - a|q|^2. 
\] 
Setting $\gamma:=2^j+1+3/2^{j+2}$, 
\[ 
H_{q,p}(0) = -\frac{3}{2^{j+2}} - p' \cdot (p-p') \cdot \frac{\gamma}{|p-p'|^2} \le -\frac{3}{2^{j+2}} + \frac{R \gamma}{\delta}, 
\] 
where $R$ and $\delta$ are positive constants (depending only on $K$ and $U_j$)
such that $|p'| \le R$ and $|p-p'| \ge \delta$. 
Also, there holds 
$|\nabla H_{q,p}| = \gamma/|p-p'|  \le \gamma/\delta$. 
Then we can conclude that $S<\infty$. 

We show that if $b>\max\{2^j, S\}$ then (1-6) holds, 
i.e. $F_b(q,p)>2^j$ if $(q,p) \not\in U_j$. 
We consider two cases:
\begin{itemize} 
\item The case $K_q \ne \emptyset$. 
In this case $H_{q,p} \in \mca{F}(b,q)$, because $H_{q,p} \le - \frac{3}{2^{j+2}}$ on $K_q$ and 
$H_{q,p}(x) \le Q_a(q,x) + S  <  Q_a(q,x) + b$ for any $x \in T^*_q\R^n$. 
Hence $F_b(q,p) \ge H_{q,p}(p) = 2^j+1>2^j$. 
\item The case $K_q = \emptyset$. 
In this case, $F_b(q,p) = Q_a(q,p)+b \ge b > 2^j$. 
\end{itemize} 
In the rest of the proof we take and fix $b$ so that (1-5) and (1-6) hold. 

\textbf{Step 2.} 
Let us take $\rho \in C^\infty_c(\R^n, \R_{\ge 0})$ such that 
$\rho(x) = \rho(-x)$ and 
$\int_{\R^n} \rho(x) \, dx = 1$. 
For any $\ep>0$ let 
$\rho^\ep(x) := \ep^{-n} \rho(x/\ep)$. 
Then we define $G^\ep: T^* \R^n \to \R$ by 
\[ 
G^\ep(q,p) := \int_{y \in T_q^*\R^n} F_b(q,y) \rho^\ep(p-y) \, dy. 
\] 
Then $G^\ep$ satisfies the following properties: 
\begin{itemize}
\item For every $q \in \R^n$, $G^\ep|_{T_q^*\R^n}$ is a $C^\infty$ convex function. 
\item If $q \notin \pr(K)$ then $G^\ep(q,p) = Q_a(q,p)+b + a \cdot c(\ep)$, where $c(\ep):= \int_{\R^n} |x|^2 \rho^\ep(x) \, dx$. 
\item $G^\ep$ is a compact perturbation of $Q_a + b + a \cdot c(\ep)$. 
\item $G^\ep(q,p) \ge F_b(q,p)$ for any $(q,p) \in T^*\R^n$. In particular, $G^\ep(q,p) > \max\{ -\frac{1}{2^j}, H_{j-1}(q,p) \}$ for any $(q,p) \in T^*\R^n$, 
and $G^\ep(q,p)>2^j$ for any $(q,p) \notin U_j$. 
\end{itemize}
Moreover, by (1-4), if $\ep$ is sufficiently small then 
\[ 
(q,p) \in K \implies G^\ep(q,p) < -\frac{1}{2^{j+1}}. 
\] 
In the rest of the proof we fix such $\ep$. 

\textbf{Step 3.} 
For each $q \in \R^n$, let us define 
$H_q: T^*\R^n \to \R$ by 
\[ 
H_q(q', p) :=  G^\ep(q,p) + a (|q'|^2 - |q|^2). 
\] 
Then $H_q|_{T_{q'}^*\R^n}$ is a $C^\infty$-convex function for every $q' \in \R^n$. 
Moreover, if $q \notin \pr(K)$ then $H_q=Q_a + b + a \cdot c(\ep)$. 

For every $q \in \R^n$, 
there exists an open neighborhood of $q$ (denoted by $U_q$) 
such that the following properties hold for every $q' \in U_q$: 
\begin{itemize} 
\item $H_q(q', p) > \max \{ - \frac{1}{2^j} , H_{j-1}(q', p) \}$ for every $p \in T_{q'}^*\R^n$. 
\item $H_q(q', p)  < - \frac{1}{2^{j+1}}$ if $(q', p) \in K$. 
\item $H_q(q', p) > 2^j$ if $(q', p) \notin U_j$. 
\end{itemize} 
Moreover, if $q \notin \pr(K$) then we may take $U_q$ so that $U_q \cap \pr(K) = \emptyset$. 

Let us consider an open covering of $\R^n$, 
$\mca{U} := \{ U_q \}_{q \in \R^n}$. 
Let $\mca{V} = \{ V_i \}_{i=1}^\infty$ 
be a refinement of $\mca{U}$ which is locally finite. 
For every $i$, choose $q_i \in \R^n$ such that $V_i \subset U_{q_i}$. 
Let $(\chi_i)_i$ be a partition of $1$ with $\mca{V}$, 
i.e. 
$\chi_i \in C^\infty(\R^n, [0,1])$
and 
$\supp \chi_i \subset V_i$
for every $i \ge 1$, 
and $\sum_{i=1}^\infty \chi_i \equiv 1$. 
Then 
\[ 
H(q,p) := \sum_{i=1}^\infty  \chi_i (q)  H_{q_i}(q,p) 
\] 
is a $C^\infty$-function on $T^*\R^n$, 
and satisfies the following properties: 
\begin{itemize} 
\item $H$ is a compact perturbation of $Q_a + b + a \cdot c(\ep)$. 
\item $H$ is fiberwise convex. 
\item $H(q,p) > H_{j-1}(q,p) $ for every $(q,p) \in T^*\R^n$. 
\item $-\frac{1}{2^j}  < H(q,p) < - \frac{1}{2^{j+1}}$ if $(q,p) \in K$. 
\item $H(q,p) > 2^j$ if $(q,p) \notin U_j$. 
\end{itemize} 

The first property holds since 
$H_{q_i} \ne Q_a + b+ a \cdot c(\ep)$ 
only if $q_i \in \pr(K)$, 
and there are only finitely many such $q_i$'s. 
The other properties are straightforward. 

\textbf{Step 4.}
Let us take a sufficiently small $\delta>0$ such that $a+ \delta \notin \pi\Z$. 
Then $H_j:= H + Q_\delta$ satisfies the following properties:
\begin{itemize} 
\item $H_j$ is a compact perturbation of $Q_{a+\delta} + b + a\cdot c(\ep)$. 
\item $H_j$ is fiberwise strictly convex. 
\item $H_j(q,p) > H_{j-1}(q,p)$ for every $(q,p) \in T^*\R^n$. 
\item $- \frac{1}{2^j} < H_j(q,p) < - \frac{1}{2^{j+1}}$ for every $(q,p) \in K$. 
\item $H_j(q,p ) > 2^j$ if $(q,p) \notin U_j$. 
\end{itemize} 
The fourth property can be achieved by taking $\delta$ sufficiently small. 
The other properties are straightforward.

Finally, setting $a_j:= a+\delta$, 
the pair $(a_j, H_j)$ 
satisfies conditions (i), (ii'), (iii), (iv), (v'). 
\end{proof} 

\begin{lem}\label{limit_of_Legendrean_duals} 
Let $K$ be a compact and fiberwise convex set in $T^*\R^n$, 
and let $(H_j)_{j \ge 1}$ and $(a_j)_{j \ge 1}$ be sequences which satisfy the conditions in Lemma \ref{sequence_of_fiberwise_convex_functions}. 
For each $j$, let $L_{H_j} \in C^\infty(T\R^n)$ denote the Legendre dual of $H_j$, namely 
\[ 
L_{H_j}(q,v):= \max_{p \in T_q^*\R^n}   (p \cdot v - H_j(q,p) ) \qquad( q \in \R^n, \, v \in T_q\R^n). 
\] 
Then the following properties hold: 
\begin{enumerate}
\item[(i):]  $L_{H_j}(q,v) > L_{H_{j+1}}(q,v)$ for any $(q,v) \in T\R^n$ and $j \ge 1$. 
\item[(ii):] $\lim_{j \to \infty} L_{H_j}(q,v)=\begin{cases}  \max_{p \in K_q} (p \cdot v)&( q \in \pr(K)) \\ -\infty &(q \notin \pr(K)).\end{cases}$ 
\item[(iii):] $\lim_{j \to \infty} \int_{S^1} L_{H_j}(\gamma(t), \dot{\gamma}(t)) \, dt = \len_K(\gamma)$ for any $\gamma \in \Lambda$. 
\end{enumerate} 
\end{lem}
\begin{proof}
(i): For each $q \in \R^n$, 
there exists $p_0 \in T^*_q\R^n$ 
which satisfies 
$L_{H_{j+1}}(q,v) = p_0 \cdot v - H_{j+1}(q, p_0)$. Then
\[ 
L_{H_j}(q,v) = \max_p (p \cdot v - H_j(q,p)) \ge p_0 \cdot v - H_j(q,p_0) > p_0 \cdot v - H_{j+1}(q, p_0) =  L_{H_{j+1}}(q, v). 
\] 

(ii) follows from 
Lemma \ref{limit_of_dual} 
applied to $(H_j|_{T_q^*\R^n})_j$, 
identifying $\R^n$ and $T_q^*\R^n$ 
via the standard Riemannian metric on $\R^n$. 

(iii): 
First, we consider the case $\gamma(S^1) \subset \pr(K)$. 
By Lemma \ref{properties_of_len_K} (i), 
$\rho_\gamma: S^1 \to \R; \, t \mapsto \max_{p \in K_{\gamma(t)}} p \cdot \dot{\gamma}(t)$ is integrable. 
On the other hand, 
$L_{H_1}(\gamma, \dot{\gamma})$ is integrable (since $\dot{\gamma}$ is square-integrable), 
and 
$(L_{H_j}(\gamma, \dot{\gamma}))_j$ is a decreasing sequence of integrable functions, 
which converges to $\rho_\gamma$ pointwise as $j \to \infty$. 
Then, by Lebesgue's dominated convergence theorem, we obtain 
\[ 
\lim_{j \to \infty} \int_{S^1} L_{H_j} (\gamma(t), \dot{\gamma}(t)) \, dt
= 
\int_{S^1} \lim_{j \to \infty} L_{H_j}(\gamma(t), \dot{\gamma}(t))\, dt 
= 
\int_{S^1}  \rho_\gamma(t) \, dt 
= \len_K(\gamma). 
\] 

Next, we consider the case $\gamma(S^1) \not\subset \pr(K)$. 
In this case $I:= \gamma^{-1}(\R^n \setminus \pr(K))$ is a nonempty open set in $S^1$. 
Now consider an obvious inequality 
\[ 
\int_{S^1} L_{H_j}(\gamma(t), \dot{\gamma}(t)) \, dt \le 
\int_{S^1 \setminus I } L_{H_1}(\gamma(t), \dot{\gamma}(t))\, dt + 
\int_{I} L_{H_j}(\gamma(t), \dot{\gamma}(t))\, dt. 
\] 
The first term on the RHS does not depend on $j$, 
and the second term goes to $-\infty$ as $j \to \infty$. 
Thus the LHS goes to $-\infty$. 
\end{proof} 

\begin{lem}\label{limit_of_dual}
Let $K$ be any compact and convex set in $\R^n$, which may be empty. 
Let $(a_j)_{j \ge 1}$ and $(h_j)_{j \ge 1}$ be sequences with the following properties: 
\begin{enumerate} 
\item[(i):]  $(a_j)_j$ is a strictly increasing sequence in $\R_{>0}$. 
\item[(ii):] $\lim_{j \to \infty} a_j  = \infty$. 
\item[(iii):] $(h_j)_j$ is a strictly increasing sequence of convex $C^\infty$-functions on $\R^n$.  
\item[(iv):] For every $j$, there exists $b_j \in \R$ such that $h_j(x)-a_j|x|^2 -b_j$ is compactly supported. 
\item[(v):] $\lim_{j \to \infty} h_j(x) =\begin{cases}\infty&(x \notin K)\\0&(x \in K).\end{cases}$
\end{enumerate} 
Then, for any $x \in \R^n$ 
\[ 
\lim_{j \to \infty}  \big( \max_{y \in \R^n}  (x \cdot y - h_j(y)) \big) =  \begin{cases}  \max_{y \in K} (x \cdot y) &(K \ne \emptyset)  \\ -\infty &(K = \emptyset). \end{cases}
\] 
\end{lem}
\begin{proof}
First we consider the case $K \ne \emptyset$. 
Let $\mca{H}$ denote the set of $h \in C^\infty(\R^n)$
with the following properties: 
\begin{itemize}
\item[(a):]  $h$ is convex. 
\item[(b):] There exists $Q \in C^\infty(\R^n)$ of the form 
\[
Q(x_1, \ldots, x_n)= \sum_{1\le i, j \le n} a_{ij} x_ix_j + b 
\] 
where $(a_{ij})_{1\le i, j \le n}$ is a non-negative symmetric matrix, 
such that $h(x)- Q(x)$ is compactly supported. 
\item[(c):] $h(x)<0$ for any $x \in K$. 
\end{itemize} 
Then the sequence $(h_j)_j$ is cofinal in $\mca{H}$, 
which implies that for any $x \in \R^n$ 
\[
\lim_{j \to \infty}\big( \max_{y \in \R^n}  (x \cdot y - h_j(y)) \big) = \inf_{h \in \mca{H}} (\max_{y \in \R^n} (x \cdot y- h(y))).
\]
For any $h \in \mca{H}$, there holds
\[ 
\max_{y \in \R^n} (y \cdot x - h(y))\ge
\max_{y \in K} (x \cdot y - h(y))\ge
\max_{y \in K} x \cdot y. 
\] 
Thus $\inf_{h \in \mca{H}} (\max_{y \in \R^n} (x \cdot y - h(y))) \ge \max_{y \in K} x \cdot y$. 
To complete the proof, it is sufficient to prove the opposite inequality, i.e. 
$\inf_{h \in \mca{H}} (\max_{y \in \R^n} (x \cdot y- h(y))) \le  \max_{y \in K} x \cdot y$. 
To prove this, 
it is sufficient to show that for any $\delta>0$
there exists $h \in \mca{H}$ such that 
\[ 
\max_{y \in \R^n} (x \cdot y - h(y)) \le \max_{y \in K} x \cdot y + \delta. 
\] 
When $x=0$ it is easy to see. 
When $x \ne 0$, 
let $I_x:= \{ x \cdot y \mid y \in K \}$ and $M_x:= \max I_x = \max_{y \in K} x\cdot y$. 
It is easy to see that there exists $\ph \in C^\infty(\R)$ with the following properties: 
\begin{itemize} 
\item $\ph$ is convex. 
\item There exist $a>0$ and $b \in \R$ such that $\ph(t)-(at^2+b)$ is compactly supported. 
\item $-\delta  \le \ph(t) < 0$ for any $t \in I_x$. 
\item $\ph'(M_x)=1$. 
\end{itemize} 
Take such $\ph$ and let $h(y):= \ph(x \cdot y)$. 
Then $h \in \mca{H}$, and there holds 
\[ 
\max_{y \in \R^n}(x \cdot y - h(y))
= 
\max_{t \in \R} (t - \ph(t)) 
= 
M_x - \ph(M_x) 
\le 
M_x + \delta. 
\] 
This completes the proof when $K \ne \emptyset$. 

Finally we consider the case $K=\emptyset$. 
Let $\mca{H}'$ denote the set of $h \in C^\infty(\R^n)$ which satisfies conditions (a) and (b) above. 
Then, the sequence $(h_j)_j$ is cofinal in $\mca{H}'$, which implies that for any $x \in \R^n$ 
\[
\lim_{j \to \infty}\big( \max_{y \in \R^n}  (x \cdot y - h_j(y)) \big) = \inf_{h \in \mca{H}'} (\max_{y \in \R^n} (x \cdot y- h(y))).
\]
If $h \in \mca{H}'$ then $h+c \in \mca{H}'$ for any $c \in \R$, thus the RHS is obviously equal to $-\infty$. This completes the proof. 
\end{proof}

\section{Proof of Theorem \ref{SH_and_loop_space_homology}}

The goal of this section is to prove Theorem \ref{SH_and_loop_space_homology}. 
In Section 4.1, we summarize basic properties of Lagrangian action functionals on the free loop space of $\R^n$. 
In Section 4.2 we state Theorem \ref{HF_and_loop_space_homology}, which shows an isomorphism between Hamiltonian Floer homology on $T^*\R^n$ and homology of loop spaces of $\R^n$. 
The proof of Theorem \ref{HF_and_loop_space_homology} occupies Sections 4.3--4.5; 
the plan of the proof of Theorem \ref{HF_and_loop_space_homology} is explained in the last paragraph of Section 4.2. 
Finally, in Section 4.6, we prove Theorem \ref{SH_and_loop_space_homology} by taking a limit of isomorphisms obtained by Theorem \ref{HF_and_loop_space_homology}. 

\subsection{Lagrangian action functional on the loop space}

Consider the following conditions (L1), (L2) for $L \in C^\infty(S^1 \times T\R^n)$:

\begin{itemize} 
\item[(L1):]  
There exist $a \in \R_{>0}$ and $b \in \R$ such that the function on $S^1 \times T\R^n$ 
\[ 
L(t,q,v)-\bigg(\frac{|v|^2}{4a}-a|q|^2+b\bigg) 
\]
is compactly supported. 
\item[(L2):] 
There exists $c\in\R_{>0}$ such that 
$\partial_v^2 L(t,q,v) \ge c$ for any $(t,q,v) \in S^1 \times T\R^n$. 
\end{itemize} 
\begin{rem}
$\partial_v^2L(t,q,v) \ge c$ means that 
the  symmetric matrix 
$(\partial_{v_i}\partial_{v_j}L(t,q,v)-c \delta_{ij})_{1\le i,j \le n}$ 
is nonnegative, where 
$\delta_{ij}=\begin{cases}1&(i=j)\\ 0&(i \ne j).\end{cases}$
\end{rem} 

Recall $\Lambda:= L^{1,2}(S^1, \R^n)$. 
If $L$ satisfies the condition (L1), then one can define the functional
$\mca{S}_L:\Lambda\to\R$ by 
\[ 
\mca{S}_L(\gamma):= \int_{S^1} L(t, \gamma(t), \dot{\gamma}(t)) \, dt. 
\] 

\begin{lem}\label{Lagrangian_action_functional} 
If $L \in C^\infty(S^1 \times T\R^n)$ satisfies (L1) and (L2), 
the functional $\mca{S}_L$ satisfies the following properties: 
\begin{itemize} 
\item[(i):] $\mca{S}_L$ is a Fr\'{e}chet $C^1$-function. 
The differential $d\mca{S}_L$ is given by 
\[ 
d\mca{S}_L(\xi):= \int_{S^1} \partial_q L(t,\gamma(t), \dot{\gamma}(t)) \cdot \xi(t) + \partial_vL(t, \gamma(t), \dot{\gamma}(t)) \cdot \dot{\xi}(t) \,dt
\qquad
(\forall \xi \in \Lambda). 
\] 
Moreover $d\mca{S}_L$ is G\^{a}teaux differentiable. 
\item[(ii):] $\gamma \in \Lambda$ satisfies $d\mca{S}_L(\gamma)=0$ if and only if $\gamma \in C^\infty(S^1, \R^n)$ and satisfies 
\[ 
\partial_qL(t, \gamma(t), \dot{\gamma}(t)) - \partial_t(\partial_vL(t, \gamma(t), \dot{\gamma}(t))) = 0. 
\] 
\item[(iii):] 
For every $\gamma \in \Lambda$, let us define $D\mca{S}_L(\gamma) \in \Lambda$ so that 
\[ 
\langle D\mca{S}_L(\gamma), \xi \rangle_{L^{1,2}} = d\mca{S}_L(\gamma)(\xi)  \qquad (\forall \xi \in \Lambda), 
\] 
where $\langle \, , \, \rangle_{L^{1,2}}$ is defined by 
$\langle f, g \rangle_{L^{1,2}}: = \int_{S^1} f(t) \cdot g(t) + \dot{f}(t) \cdot \dot{g}(t) \, dt$. 
Then the pair $(\mca{S}_L, D \mca{S}_L)$ satisfies the Palais-Smale condition. 
Namely, if a sequence $(x_k)_k$ on $\Lambda$ 
satisfies 
$\sup_k  |\mca{S}_L(x_k)| < \infty$ 
and 
$\lim_{k \to \infty}  d\mca{S}_L(D\mca{S}_L(x_k))=0$
then 
$(x_k)_k$ contains a convergent subsequence. 
\end{itemize} 
\end{lem}
\begin{proof}
(i) and (ii) follow from Proposition 3.1 (i), (ii) of \cite{Abbondandolo_Schwarz_pseudo_gradient}. 
(iii) is proved as Corollary 3.4 of \cite{JSG}, 
which is based on Proposition 3.3 of \cite{Abbondandolo_Schwarz_pseudo_gradient}. 
\end{proof} 

Suppose that $L \in C^\infty(S^1 \times T\R^n)$ satisfies (L1) and (L2). 
Let $\mca{P}(L)$ denote the set of critical points of $\mca{S}_L$, namely
\[ 
\mca{P}(L):= \{\gamma \in \Lambda \mid d\mca{S}_L(\gamma)=0\}. 
\] 
For any $\gamma \in \mca{P}(L)$, 
the second G\^{a}teaux differential $d^2\mca{S}_L(\gamma)$ 
is Fredholm and has finite Morse index (see Proposition 3.1 (iii) of \cite{Abbondandolo_Schwarz_pseudo_gradient}). 
The Morse index is denoted by $\ind_\Morse(\gamma)$. 
We say that $\gamma$ is \textit{nondegenerate} 
if $0$ is not an eigenvalue of $d^2\mca{S}_L(\gamma)$. 
Let us introduce the following condition for $L \in C^\infty(S^1 \times T\R^n)$: 
\begin{itemize}
\item[(L0):] 
Every $\gamma \in \mca{P}(L)$ is nondegenerate. 
\end{itemize} 

\subsection{Isomorphism between Hamiltonian Floer homology and loop space homology} 

Let us consider the following condition for $H \in C^\infty(S^1\times T^*\R^n)$: 
\begin{itemize}
\item[(H2):] 
There exists $c \in \R_{>0}$ such that 
$\partial_p^2H(t,q,p)\ge c$ for any $(t,q,p) \in S^1\times T^*\R^n$.
\end{itemize} 
For any $H \in C^\infty(S^1\times T^*\R^n)$ which satisfies (H1) and (H2), 
its Legendre dual $L_H \in C^\infty(S^1 \times T\R^n)$ is defined by 
\[ 
L_H(t,q,v):= \max_{p \in T_q^*\R^n}   (p \cdot v - H(t,q,p)) \qquad (t \in S^1, \, q \in \R^n, \, v \in T_q\R^n). 
\] 

\begin{lem} \label{properties_of_L_H}
\begin{enumerate} 
\item[(i):] 
If $H$ satisfies (H1) and (H2), then $L_H$ satisfies (L1) and (L2). 
Moreover, the map 
\[ 
\mca{P}(H) \to \mca{P}(L_H); \, x \mapsto \gamma_x:= \pr \circ x 
\] 
is a bijection, and the inverse map is 
\[ 
\mca{P}(L_H)  \to \mca{P}(H); \, \gamma \mapsto (\gamma, p_\gamma)
\] 
where $p_\gamma$ is characterized by 
\[ 
L_H(t, \gamma(t), \dot{\gamma}(t)) = 
p_\gamma(t) \cdot \dot{\gamma}(t) - H(t, \gamma(t), \dot{\gamma}(t))  \qquad(\forall t \in S^1). 
\] 
\item[(ii):] 
In the situation of (i), 
for any $x \in \mca{P}(H)$, $\gamma_x$ is nondegenerate if and only if $x$ is nondegenerate. 
Moreover, for any such $x$, there holds 
$\ind_\Morse(\gamma_x)= \ind_\CZ(x)$. 
\end{enumerate} 
\end{lem}
\begin{proof}
(i) can be checked by direct computations. 
(ii) follows from Theorem 1 of \cite{Long} Section 7.3. 
\end{proof} 

\begin{rem} 
Lemma \ref{properties_of_L_H} (ii) extends to Hamiltonians on arbitrary manifolds, 
at least when $H$ is a ``classical'' Hamiltonian (i.e. the sum of the kinetic energy and a potential function on the base) on a Riemannian manifold $M$, although one needs a correction term
if the vector bundle $\gamma_x^*TM$ is not oriented. See Theorem 1.2 and Lemma 2.1 of \cite{Weber}. 
\end{rem} 

Now let us state the isomorphism between Hamiltonian Floer homology on $T^*\R^n$ and homology of loop spaces of $\R^n$: 

\begin{thm}\label{HF_and_loop_space_homology} 
For any $H \in C^\infty(S^1\times T^*\R^n)$ which satisfies (H0), (H1), (H2),
and any real numbers $a<b$, 
one can define an isomorphism 
\[ 
\HF^{[a,b)}_*(H) \cong H_*( \mca{S}_{L_H}^{-1}(\R_{<b}), \mca{S}_{L_H}^{-1}(\R_{<a}))
\]
so that the following diagram commutes: 
\begin{equation}\label{Haba'b'}
\xymatrix{
\HF^{[a,b)}_*(H) \ar[r] \ar[d]_-{\cong} & \HF^{[a',b')}_*(H) \ar[d]^-{\cong} \\
H_*(\mca{S}_{L_H}^{-1}(\R_{<b}), \mca{S}_{L_H}^{-1}(\R_{<a}))
\ar[r] & H_*(\mca{S}_{L_H}^{-1}(\R_{<b'}), \mca{S}_{L_H}^{-1}(\R_{<a'}))
}
\end{equation}
where $a \le a'$ and $b \le b'$, 
\begin{equation}\label{HH'ab}
\xymatrix{
\HF^{[a,b)}_*(H) \ar[r] \ar[d]_-{\cong} & \HF^{[a,b)}_*(H') \ar[d]^-{\cong} \\
H_*(\mca{S}_{L_H}^{-1}(\R_{<b}), \mca{S}_{L_H}^{-1}(\R_{<a}))
\ar[r] & H_*(\mca{S}_{L_{H'}}^{-1}(\R_{<b}), \mca{S}_{L_{H'}}^{-1}(\R_{<a}))
}
\end{equation}
where $H(t,q,p) < H'(t,q,p) \, (\forall (t,q,p)\in S^1 \times T^*\R^n)$. 
\end{thm} 
\begin{rem} 
Commutative diagrams (\ref{Haba'b'}) and (\ref{HH'ab}) are special cases of the following commutative diagram: 
\begin{equation}\label{HH'aba'b'} 
\xymatrix{
\HF^{[a,b)}_*(H) \ar[r] \ar[d]_-{\cong} & \HF^{[a',b')}_*(H') \ar[d]^-{\cong} \\
H_*(\mca{S}_{L_H}^{-1}(\R_{<b}), \mca{S}_{L_H}^{-1}(\R_{<a}))
\ar[r] & H_*(\mca{S}_{L_{H'}}^{-1}(\R_{<b'}), \mca{S}_{L_{H'}}^{-1}(\R_{<a'})), 
}
\end{equation}
where $a \le a'$, $b \le b'$ and $H(t,q,p) < H'(t,q,p)\, (\forall (t,q,p) \in S^1 \times T^*\R^n)$. 
\end{rem}

The proof of Theorem \ref{HF_and_loop_space_homology}, which follows the arguments in \cite{Abbondandolo_Schwarz_Floer_homology} and \cite{JSG}, occupies Sections 4.3--4.5. 
In Section 4.3 we recall the construction of Morse complex of Lagrangian action functionals.
In Section 4.4 we explain a chain-level construction of the isomorphism in Theorem \ref{HF_and_loop_space_homology} and check the commutativity of the diagram (\ref{Haba'b'}). 
In Section 4.5 we prove the commutativity of the diagram (\ref{HH'ab}). 

\subsection{Morse theory for Lagrangian action functionals} 

Suppose that $L \in C^\infty(S^1\times T\R^n)$ satisfies (L0), (L1) and (L2). 
The goal of this subsection is to recall the construction of the Morse complex of $\mca{S}_L$. 

For each $k \in \Z_{\ge 0}$, 
let $\CM_k(L)$ denote the free $\Z/2$-module generated over 
\[
\{\gamma \in \mca{P}(L) \mid \ind_\Morse(\gamma)=k\}.
\] 
To define the boundary operator we need the following lemma. 
For definitions of 
``Morse vector field'' and 
``Morse-Smale condition'', 
see Section 2 of \cite{Abbondandolo_Schwarz_pseudo_gradient}. 
In the next lemma, 
$\Lambda=L^{1,2}(S^1, \R^n)$
is equipped with a natural structure of a Hilbert manifold. 

\begin{lem}\label{pseudo_gradient}
If $L \in C^\infty(S^1 \times T\R^n)$ 
satisfies (L0), (L1), (L2), 
there exists a smooth vector field $X$ on $\Lambda$ which satisfies the following conditions:
\begin{enumerate}
\item[(i):] $X$ is complete. 
\item[(ii):] $\mca{S}_L$ is a Lyapunov function for $X$. Namely, $d\mca{S}_L(X(\gamma))<0$ if $X(\gamma)\ne 0$. 
\item[(iii):] $X$ is a Morse vector field. $X(\gamma)=0$ if and only if $\gamma \in \mca{P}(L)$, 
and the Morse index of $X$ at $\gamma$ is equal
to the Morse index of $\gamma$ as a critical point of $\mca{S}_L$. 
\item[(iv):] The pair $(\mca{S}_L, X)$ satisfies the Palais-Smale condition.
\item[(v):] $X$ satisfies the Morse-Smale condition up to every order. 
\end{enumerate} 
\end{lem} 
\begin{proof}
This lemma follows from Lemma 3.5 of \cite{JSG}
(which is essentially same as Theorem 4.1 of \cite{Abbondandolo_Schwarz_pseudo_gradient}), 
since the condition (L1) of \cite{JSG} is weaker than the condition (L1) of this paper. 
\end{proof} 

Let us take a vector field $X$ on $\Lambda$ which satisfies the conditions in Lemma \ref{pseudo_gradient}. 
Let $(\varphi^t_X)_{t\in\R}$ denote the flow on $\Lambda$ generated by $X$. 
For any $\gamma \in \mca{P}(L)$ let us set
\begin{align*}
W^u(\gamma:X)&:=\{x \in \Lambda \mid \lim_{t\to-\infty}\varphi^t_X(x)=\gamma \} \\ 
W^s(\gamma:X)&:=\{x \in \Lambda \mid \lim_{t\to\infty}\varphi^t_X(x)=\gamma \}. 
\end{align*} 

For any real numbers $a<b$, 
let $\CM^{[a,b)}_*(L)$ denote the free $\Z/2$-module generated over
$\{ \gamma \in \mca{P}(L) \mid a \le \mca{S}_L(\gamma) < b \}$. 
For any two generators $\gamma$ and $\gamma'$, let 
\[
\mca{M}_X(\gamma, \gamma') := W^u(\gamma:X) \cap W^s(\gamma': X). 
\] 
When $\gamma \ne \gamma'$, 
let $\bar{\mca{M}}_X(\gamma,\gamma')$ denote the quotient of  $\mca{M}_X(\gamma, \gamma')$ 
by the natural $\R$-action. 
Since $X$ satisfies the Morse-Smale condition, 
the boundary operator
\[ 
\partial_{L, X}: \CM^{[a,b)}_*(L) \to \CM^{[a,b)}_{*-1}(L); \, 
\gamma \mapsto 
\sum_{\ind_\Morse(\gamma') = \ind_\Morse(\gamma)-1}   \#_2 \bar{\mca{M}}_X(\gamma, \gamma') \cdot \gamma' 
\] 
is well-defined and satisfies $\partial_{L,X}^2=0$. 
Homology of the chain complex $(\CM_*^{[a,b)}(L), \partial_{L,X})$ 
does not depend on the choice of $X$, and denoted by $\HM_*^{[a,b)}(L)$. 
There exists a natural isomorphism 
$\HM_*^{[a,b)}(L)\cong H_*( \mca{S}_L^{-1}(\R_{<b}), \mca{S}_L^{-1}(\R_{<a}))$. 
These facts follow from Theorems 2.7, 2.8 and 2.11 in \cite{Abbondandolo_Majer}. 

Consider $L^0, L^1 \in C^\infty(S^1 \times T\R^n)$ which satisfy (L0), (L1), (L2)
and 
$L^0(t,q,v)>L^1(t,q,v)$ for any $(t,q,v) \in S^1 \times T\R^n$. 
We also assume that $\mca{P}(L^0) \cap \mca{P}(L^1) = \emptyset$.

Take vector fields $X^0, X^1$ on $\Lambda$ such that 
$(L^0,X^0)$ and $(L^1,X^1)$ satisfy the conditions in Lemma \ref{pseudo_gradient}. 
By taking small perturbations of $X^0$ and $X^1$
(note that these perturbations do not change Morse complexes of $L^0$ and $L^1$), 
we can achieve the following condition: 
\begin{quote}
For any $\gamma^0 \in \mca{P}(L^0)$ and $\gamma^1 \in \mca{P}(L^1)$, 
$W^u(\gamma^0:X^0)$ is transverse to $W^s(\gamma^1 : X^1)$. 
\end{quote}
If this assumption is satisfied, 
$\mca{M}_{X^0, X^1}(\gamma^0, \gamma^1):= W^u(\gamma^0:X^0) \cap W^s(\gamma^1:X^1)$
is a smooth manifold of dimension $\ind_\Morse(\gamma^0) - \ind_\Morse(\gamma^1)$. 
Then we define a chain map
\[  
\Phi: \CM_*^{[a,b)}(L^0,X^0) \to \CM_*^{[a,b)}(L^1,X^1); \quad
\gamma 
\mapsto 
\sum_{\ind_\Morse(\gamma')=\ind_\Morse(\gamma)} \sharp_2 
\mca{M}_{X^0, X^1}(\gamma, \gamma')\cdot  \gamma'. 
\]
$\Phi$ induces a linear map on homology 
$\HM_*^{[a,b)}(L^0)\to\HM_*^{[a,b)}(L^1)$, 
which does not depend on the choices of $X^0$, $X^1$. 
Via isomorphisms between the Morse homology and the loop space homology, 
this map corresponds to the map 
\[ 
H_*(\mca{S}_{L^0}^{-1}(\R_{<b}), \mca{S}_{L^0}^{-1}(\R_{<a})) \to  H_*(\mca{S}_{L^1}^{-1}(\R_{<b}), \mca{S}_{L^1}^{-1}(\R_{<a}))
\] 
which is induced by the inclusion map. 

\subsection{Isomorphism at chain level} 

Let us take $H \in C^\infty(S^1\times T^*\R^n)$ satisfying (H0), (H1), (H2). 
Its Legendre dual $L_H$ satisfies (L0), (L1), (L2) by Lemma \ref{properties_of_L_H}. 
Let us also take real numbers $a<b$. 
The goal of this subsection is to define a chain map 
$\CM^{[a,b)}_*(L_H) \to \CF^{[a,b)}_*(H)$
which induces an isomorphism 
$\HM^{[a,b)}_*(L_H) \cong \HF^{[a,b)}(H)$. 

The definition of the chain map involves
``hybrid moduli spaces'' 
introduced by Abbondandolo-Schwarz \cite{Abbondandolo_Schwarz_Floer_homology}. 
Let us take $X$ and $J$ as follows: 
\begin{itemize}
\item $X$ is a vector field on $\Lambda$ such that $\CM^{[a,b)}_*(L_H,X)$ is well-defined. 
\item $J=(J_t)_{t\in S^1}$ is a family of almost complex structures on $T^*\R^n$ such that $\CF^{[a,b)}_*(H,J)$ is well-defined.
\end{itemize}

For any $\gamma \in \mca{P}(L_H)$ with $\mca{S}_{L_H}(\gamma) \in [a,b)$ and 
$x \in \mca{P}(H)$ with $\mca{A}_H(x) \in [a,b)$, 
let $\mca{M}_{X,H,J}(\gamma, x)$ denote the set of 
$u\in L^{1,3}(\R_{\ge 0} \times S^1,  T^*\R^n)$ such that 
\begin{align*}
&\partial_s u - J_t(\partial_t u - X_{H_t}(u))=0, \\
&\pr \circ u_0 \in W^u(\gamma: X),\\
&\lim_{s\to\infty}u_s=x. 
\end{align*} 
Here $u_s: S^1 \to T^*\R^n$ is defined by 
$u_s(t):= u(s,t)$. 

\begin{rem} 
The above Sobolev space $L^{1,3}$ can be replaced with $L^{1,r}$ for any $2<r \le 4$; 
see pp.299 of \cite{Abbondandolo_Schwarz_Floer_homology}. 
\end{rem} 

\begin{lem}\label{Action_hybrid_1} 
Let $\gamma$ and $x$ be as above. 
\begin{enumerate} 
\item[(i):] 
For any $u\in\mca{M}_{X,H,J}(\gamma, x)$, there holds
\[ 
\mca{S}_{L_H}(\gamma) \ge \mca{S}_{L_H}(\pr \circ u_0) \ge \mca{A}_H(u_0) \ge \mca{A}_H(x). 
\] 
In particular, if $\mca{M}_{X,H,J}(\gamma, x) \ne \emptyset$ then $\mca{S}_{L_H}(\gamma) \ge \mca{A}_H(x)$. 
\item[(ii):] 
If $\mca{S}_{L_H}(\gamma)=\mca{A}_H(x)$, 
then $\mca{M}_{X,H,J}(\gamma, x) \ne \emptyset$
if and only if $x = \pr \circ \gamma$. 
Moreover, the moduli space $\mca{M}_{X,H,J}(\gamma, \pr \circ \gamma)$ 
consists of a point which is cut out transversally. 
\end{enumerate}
\end{lem}
\begin{proof} 
See pp.299 of \cite{Abbondandolo_Schwarz_Floer_homology} for (i) and the first sentence in (ii). 
For the second sentence in (ii), see Proposition 3.7 of \cite{Abbondandolo_Schwarz_Floer_homology}. 
\end{proof}

\begin{lem}\label{Fredholm_hybrid_1} 
For generic $J$, 
$\mca{M}_{X, H, J}(\gamma,x)$ 
has a structure of a $C^\infty$-manifold of dimension 
$\ind_{\Morse}(\gamma)-\ind_{\CZ}(x)$ 
for any $\gamma$ and $x$ as above. 
\end{lem} 
\begin{proof}
The case $x = \pr \circ \gamma$ is discussed in Lemma \ref{Action_hybrid_1} (ii). 
The other cases follow from the standard argument using \cite{Floer_Hofer_Salamon}. 
See pp.313 of \cite{Abbondandolo_Schwarz_Floer_homology}. 
\end{proof} 

Let us state the following $C^0$-estimate. 
For comments on the proof see Remark \ref{rem_C0_estimates}. 

\begin{lem}\label{C0_hybrid_1} 
If $\sup_{t \in S^1} \|J_t - J_\std \|_{C^0}$ is sufficiently small, 
then for any $\gamma$ and $x$ as above 
\[ 
\sup_{\substack{u\in\mca{M}_{X,H,J}(\gamma,x)\\(s,t)\in\R_{\ge 0} \times S^1}}  |u(s,t)| < \infty.
\]
\end{lem} 

By these results 
and the standard compactness and glueing arguments 
(see Sections 3.3 and 3.4 of \cite{Abbondandolo_Schwarz_Floer_homology}), 
generic $J$ which is sufficiently close to $J_\std$ satisfies the following properties: 
\begin{itemize}
\item For any $\gamma$ and $x$ as above satisfying 
$\ind_\Morse(\gamma)-\ind_\CZ(x)=0$, 
the moduli space 
$\mca{M}_{X,H,J}(\gamma,x)$ is a finite set. 
\item A linear map 
\[ 
\Psi: \CM^{[a,b)}_*(L,X)\to\CF^{[a,b)}_*(H,J);  \gamma \mapsto \sum_{\ind_\CZ(x)=\ind_\Morse(\gamma)}  \#_2 \mca{M}_{X, H, J}(\gamma, x) \cdot x 
\]
is a chain map with respect to boundary operators $\partial_{L_H, X}$ and $\partial_{H,J}$. 
\end{itemize} 
Finally, Lemma \ref{Action_hybrid_1} implies that $\Psi$ is an isomorphism (see Section 3.5 of \cite{Abbondandolo_Schwarz_Floer_homology}). 
In particular, $H_*(\Psi): \HM^{[a,b)}_*(L) \to \HF^{[a,b)}_*(H)$ is an isomorphism. 

For any $a, b, a', b' \in \R$ 
satisfying $a<b$, $a'<b'$, $a\le a'$ and $b\le b'$, 
the commutativity of the following diagram is straightforward from the definition of $\Psi$: 
\begin{equation}
\xymatrix{
\CM^{[a,b)}_*(L_H) \ar[r] \ar[d]_-{\cong} & \CM^{[a',b')}_*(L_H) \ar[d]^-{\cong} \\
\CF^{[a,b)}_*(H) \ar[r]&\CF^{[a',b')}_*(H). 
}
\end{equation}
This implies the commutativity of (\ref{Haba'b'}). 

\subsection{Commutativity of monotonicity maps} 

The goal of this subsection is to prove the commutativity of (\ref{HH'ab}). 
Let us take the following data:
\begin{itemize}
\item
$H, H' \in C^\infty(S^1\times T^*\R^n)$ 
satisfying (H0), (H1), (H2) and 
$H(t,q,p)<H'(t,q,p)$ for any $(t,q,p) \in S^1\times T^*\R^n$. 
\item 
Real numbers $a<b$. 
\item 
Almost complex structures $J$, $J'$ 
and vector fields $X$, $X'$ such that 
chain complexes 
$\CF_*^{[a,b)}(H,J)$, 
$\CF_*^{[a,b)}(H', J')$, 
$\CM_*^{[a,b)}(L_H, X)$, 
$\CM_*^{[a,b)}(L_{H'}, X')$
are defined. 
\end{itemize} 

Without loss of generality, we may assume $\mca{P}(L_H) \cap \mca{P}(L_{H'})=\emptyset$.
Indeed, for any $H$ and $H'$ 
satisfying (H0), (H1), (H2) and $H<H'$, 
there exists a strictly increasing sequence $(H_j)_{j \ge 1}$
such that every $H_j$ satisfies (H0), (H1), (H2), 
$\lim_{j \to \infty} H_j = H$,
and 
\[ 
\mca{P}(L_{H_j}) \cap \mca{P}(L_{H_{j+1}}) = \emptyset, \quad
\mca{P}(L_{H_j}) \cap \mca{P}(L_H) = \emptyset,  \quad
\mca{P}(L_{H_j}) \cap \mca{P}(L_{H'}) = \emptyset
\]
for every $j \ge 1$. 
Then, assuming that the commutativity of (\ref{HH'ab})
is proved for pairs $(H_j, H_{j+1})$, $(H_j, H)$ and $(H_j, H')$ for every $j$, 
the commutativity of (\ref{HH'ab}) for $(H, H')$ follows by taking limits. 

In the previous subsection we defined isomorphisms of chain complexes 
$\Psi: \CM_*^{[a,b)}(L_H, X) \to  \CF_*^{[a,b)}(H, J)$ 
and
$\Psi': \CM_*^{[a,b)}(L_{H'}, X') \to \CF_*^{[a,b)}(H', J')$. 
We also defined chain maps 
$\Phi^L: \CM_*^{[a,b)}(L_H, X) \to \CM_*^{[a,b)}(L_{H'}, X')$ 
and 
$\Phi^H: \CF_*^{[a,b)}(H, J) \to \CF_*^{[a,b)}(H', J')$. 
Our goal is to show that the following diagram commutes up to homotopy: 
\begin{equation}\label{HH'ab_chainlevel} 
\xymatrix{
\CM^{[a,b)}_*(L_H, X)\ar[r]^-{\Psi}_-{\cong} \ar[d]_-{\Phi^L}&\CF^{[a,b)}_*(H,J)\ar[d]^-{\Phi^H}\\
\CM^{[a,b)}_*(L_{H'},X')\ar[r]_-{\Psi'}^-{\cong}&\CF^{[a,b)}_*(H',J').
}
\end{equation}
This immediately implies the commutativity of the diagram (\ref{HH'ab}). 
Since vector spaces in the diagram (\ref{HH'ab_chainlevel})
are generated by finitely many critical points, 
boundary operators and chain maps in this diagram do not change under $C^\infty$-small perturbations of $X$, $X'$, $J$, $J'$. 
Hence we may assume that these data are taken so that 
all moduli spaces which appear in the rest of this subsection are cut out transversally. 

To prove that (\ref{HH'ab_chainlevel}) commutes up to homotopy, 
first we define a linear map 
\[
\Theta: \CM^{[a,b)}_*(L_H,X) \to \CF^{[a,b)}_*(H',J'); \quad
\gamma \mapsto \sum_{\ind_\Morse(\gamma)=\ind_\CZ(x)} \#_2 \mca{M}_{X,H',J'}(\gamma, x) \cdot x. 
\] 
$\Theta$ is a chain map (namely $\partial_{H',J'} \circ \Theta = \Theta \circ \partial_{L_H,X}$)
by the same reason that $\Psi$ in the previous subsection is a chain map. 
We are going to prove 
$\Phi^H \circ \Psi \sim \Theta \sim \Psi' \circ \Phi^L$.

First we prove $\Psi' \circ \Phi^L \sim \Theta$. 
For any $\gamma \in \mca{P}(L_H)$ and $x \in \mca{P}(H')$
such that $\mca{S}_{L_H}(\gamma), \mca{A}_{H'}(x) \in [a,b)$, 
let $\mca{N}^0(\gamma,x)$ denote the set of $(\alpha, u, v)$ where 
\[
\alpha \in \R_{\ge 0}, \quad u: [0,\alpha] \to \Lambda,  \quad v \in L^{1,3}(\R_{\ge 0} \times S^1, T^*\R^n)
\]
such that
\begin{align*}
&u(0) \in W^u(\gamma:X), \quad  u(s) = \ph^s_{X'}(u(0)) \,(\forall s \in [0,\alpha]),  \\
&\partial_s v - J'_t(\partial_t v - X_{H'_t}(v))=0,  \\
&\pr \circ v_0 = u(\alpha), \quad   \lim_{s\to \infty}v_s= x. 
\end{align*} 

Let us state the following $C^0$-estimate: 

\begin{lem}\label{C0_hybrid_2}
If $\sup_{t \in S^1} \|J_t-J_\std\|_{C^0}$ is sufficiently small, 
then for any $\gamma$ and $x$ as above 
\[ 
\sup_{\substack{(\alpha,u,v)\in\mca{N}^0(\gamma,x)\\(s,t) \in \R_{\ge 0} \times S^1}}   |v(s, t)| < \infty. 
\] 
\end{lem} 

For generic $J'$ which is sufficiently close to $J_\std$, 
$\mca{N}^0(\gamma, x)$ is a finite set for any 
$\gamma$ and $x$ 
satisfying $\ind_\CZ(x) = \ind_\Morse(\gamma)+1$, 
and the linear map 
\[ 
K^0: \CM^{[a,b)}_*(L_H)  \to \CF^{[a,b)}_{*+1}(H'); \quad \gamma \mapsto \sum_{\ind_{\CZ}(x) = \ind_{\Morse}(\gamma)+1} \#_2\mca{N}^0(\gamma, x) \cdot x 
\] 
satisfies 
$\partial_{H', J'} \circ K^0 + K^0 \circ \partial_{L_H, X}= \Theta - \Psi' \circ \Phi^L$. 
For details see Section 4.3 of \cite{JSG}. 

Secondly we prove $\Phi^H \circ \Psi \sim \Theta$. 
Let us take $(H_{s,t})_{(s,t)\in\R \times S^1}$ and 
$(J_{s,t})_{(s,t)\in\R \times S^1}$ 
which satisfy (HH1), (HH2), (HH3) and (JJ1), (JJ2). 
In particular there exists $s_2>0$ such that 
\[ 
(H_{s,t}, J_{s,t}) =\begin{cases}(H_t, J_t) &(s \le -s_2)\\(H'_t, J'_t) &(s \ge s_2).\end{cases}
\] 
For any $\gamma \in \mca{P}(L_H)$ and $x \in \mca{P}(H')$ 
such that $\mca{S}_{L_H}(\gamma), \mca{A}_{H'}(x) \in [a, b)$, 
let $\mca{N}^1(\gamma, x)$ denote the set of $(\beta, w)$ where 
\[ 
\beta \in \R_{\le s_2}, \qquad w \in L^{1,3}(\R_{\ge \beta} \times S^1, T^*\R^n)
\] 
such that 
\begin{align*} 
&\pr \circ w_\beta \in W^u(\gamma: X), \qquad \partial_s w - J_{s,t}(\partial_t w - X_{H_{s,t}}(w))=0 \\
&\lim_{s\to \infty} w_s = x. 
\end{align*} 

Let us state the following $C^0$-estimate: 

\begin{lem}\label{C0_hybrid_3} 
If $\sup_{t\in S^1}\|J_t-J_\std\|_{C^0}$ is sufficiently small, 
then for any $\gamma$ and $x$ as above 
\[ 
\sup_{\substack{(\beta,w)\in\mca{N}^1(\gamma,x)\\(s,t)\in\R_{\ge\beta}\times S^1}} |w(s,t)|<\infty.
\] 
\end{lem} 

For generic $J$ which is sufficiently close to $J_{\std}$, 
$\mca{N}^1(\gamma,x)$ is a finite set for any 
$\gamma$ and $x$ 
satisfying $\ind_\CZ(x) = \ind_\Morse(\gamma)+1$, 
and the linear map 
\[ 
K^1: \CM^{[a,b)}_*(L) \to \CF^{[a,b)}_{*+1}(H'); \quad \gamma \mapsto \sum_{\ind_{\CZ}(x) = \ind_{\Morse}(\gamma)+1} \#_2\mca{N}^1(\gamma, x) \cdot x 
\] 
satisfies 
$\partial_{H', J'} \circ K^1 + K^1 \circ \partial_{L, X}= \Theta - \Phi^H \circ \Psi$. 
For details see Section 4.3 of  \cite{JSG}. 

\begin{rem}[Proofs of $C^0$-esimtates]\label{rem_C0_estimates}
$C^0$-estimates in this section, namely 
Lemmas
\ref{C0_hybrid_1}, 
\ref{C0_hybrid_2}, 
\ref{C0_hybrid_3}, 
are slight generalizations of 
Lemmas 4.8, 4.9, 4.10 in \cite{JSG}. 
These results in \cite{JSG} 
are stated for Hamiltonians of special type
(i.e. elements of the sequence $(H^m)_m$ defined in Section 4.1 of \cite{JSG}), 
however the proofs of these results in \cite{JSG} 
use only assumptions 
(JJ1), (JJ2), (HH1), (HH2), (HH3). 
Hence the proofs in \cite{JSG} work without any modification 
for Lemmas \ref{C0_hybrid_1}, \ref{C0_hybrid_2}, \ref{C0_hybrid_3}. 
Strictly speaking, the condition (HH3) in \cite{JSG} requires 
$b(s) \equiv 0$ in the condition (HH3) in this paper. 
Namely, if $H \in C^\infty(\R \times S^1 \times T^*\R^n)$ 
satisfies (HH3) in this paper, 
then there exists $b \in C^\infty(\R)$ such that 
\begin{equation}\label{H0Hbs} 
H^0(s,t,q,p):= H(s,t,q,p) - b(s)
\end{equation}
satisfies the condition (HH3) in \cite{JSG}. 
However, this difference does not affect Floer equations, since
(\ref{H0Hbs}) obviously implies $X_{H_{s,t}}(q,p) = X_{H^0_{s,t}}(q,p)$
for any $(s,t) \in \R \times S^1$ and $(q, p) \in T^*\R^n$. 
\end{rem} 

\subsection{Proof of Theorem \ref{SH_and_loop_space_homology}}

Now we can complete the proof of Theorem \ref{SH_and_loop_space_homology}.
Let $K$ be any nonempty, compact and fiberwise convex set in $T^*\R^n$. 
Taking time-dependent perturbations of 
Hamiltonians obtained in Lemma \ref{sequence_of_fiberwise_convex_functions}, 
there exists a sequence $(H_j)_{j \ge 1}$ in $C^\infty(S^1 \times T^*\R^n)$ which satisfies the following conditions:
\begin{itemize}
\item $H_j$ satisfies (H0), (H1), (H2) for every $j \ge 1$. 
\item $H_j(t,q,p)<H_{j+1}(t,q,p)$ for every $j \ge 1$ and $(t,q,p) \in S^1 \times T^*\R^n$. 
\item $\lim_{j\to\infty}H_j(t,q,p)=\begin{cases}0&((q,p)\in K)\\ \infty&((q,p)\notin K) \end{cases}$ for any $(t,q,p) \in S^1 \times T^*\R^n$. 
\end{itemize}

For each $j$, 
let $L_j:=L_{H_j} \in C^\infty(S^1\times T\R^n)$ denote the Legendre dual of $H_j$. 
Then, 
$(\mca{S}_{L_j}^{-1}(\R_{<c}))_{j \ge 1}$ 
is an increasing sequence of open sets in $\Lambda$
for any $c \in \R$. 
Moreover 
$\bigcup_{j=1}^\infty \mca{S}_{L_j}^{-1}(\R_{<c})=\Lambda^c_K$
by Lemma \ref{limit_of_Legendrean_duals} (iii). 
Then we obtain 
\begin{align*} 
\SH^{[a,b)}_*(K)&=\varinjlim_{j\to\infty}\HF^{[a,b)}_*(H_j)\\
&\cong \varinjlim_{j\to\infty}H_*(\mca{S}_{L_j}^{-1}(\R_{<b}),\,\mca{S}_{L_j}^{-1}(\R_{<a}))\\
&\cong H_*\bigg(\bigcup_{j=1}^\infty \mca{S}_{L_j}^{-1}(\R_{<b}), \bigcup_{j=1}^\infty \mca{S}_{L_j}^{-1}(\R_{<a})\bigg)\\
&=H_*(\Lambda^b_K, \Lambda^a_K), 
\end{align*}
where the isomorphism on the second line follows from the commutativity of (\ref{HH'ab}). 
Finally, the commutativity of (\ref{SH_aba'b'})
follows from 
the commutativity of (\ref{HH'aba'b'})
and taking limits of Hamiltonians. 
This completes the proof of Theorem \ref{SH_and_loop_space_homology}. 
\qed

\section{Proof of Theorem \ref{SH=EHZ}}

The goal of this section is to prove 
Theorem \ref{SH=EHZ}. 
Namely, we prove $c_{\SH}(K)=c_{\EHZ}(K)$ for any convex body $K \subset T^*\R^n$. 

The case $n=1$ can be proved by the following simple argument. 
For any convex body $K \subset T^*\R^1$, both $c_{\EHZ}(K)$ and the Hamiltonian displacement energy of $K$ (denoted by $e(K)$) are equal to the measure of $K$.
On the other hand, 
$c_{\EHZ}(K) \le c_{\SH}(K)$ (by Lemma \ref{properties_of_c_SH} (iii))
and 
$c_{\SH}(K) \le e(K)$ (second inequality in Theorem 1.4 of \cite{Hermann}), 
thus $c_{\EHZ}(K)=c_{\SH}(K)=e(K)$. 

Hence we assume $n \ge 2$ in the rest of the proof. 
Let us first introduce the notion of \textit{nice} convex bodies. 

\begin{defn}\label{defn_nice_convex_body} 
A convex body $K \subset T^*\R^n$ is called \textit{nice} 
if $\partial K$ is of $C^\infty$ and strictly convex, 
and there exists a $C^\infty$-map $\Gamma: S^1 \to \partial K$
which satisfies the following conditions: 
\begin{enumerate} 
\item[(i):]  $\dot{\Gamma}(t)$ generates $\ker(\omega_n|_{T_{\Gamma(t)}\partial K})$ and of positive direction
(i.e. $\omega_n(X, \dot{\Gamma}(t))>0$ for any $X \in T_{\Gamma(t)}(T^*\R^n)$ which points strictly outwards)
 for every $t \in S^1$, 
\item[(ii):] $\int_{S^1}\Gamma^*\bigg(\sum_{i=1}^n p_idq_i\bigg)=c_{\EHZ}(K)$, 
\item[(iii):] $\pr \circ \Gamma (S^1) \subset \interior (\pr(K))$. 
\end{enumerate} 
Any curve $\Gamma$ which satisfies these three conditions is called a \textit{nice curve} on $\partial K$. 
\end{defn} 

\begin{rem} 
The convex body $B:= \{ (q,p) \in T^*\R^n \mid |q|^2 + |p|^2 \le 1\}$ is not nice. 
Indeed, if $\Gamma: S^1 \to \partial B$ satisfies the conditions (i) and (ii) above, 
then $\Gamma(S^1)= \{ (e \sin t, e \cos t) \mid t \in \R/2\pi \Z\}$, thus 
$\pr(\Gamma(S^1)) = \{ es \mid -1 \le s \le 1\}$. 
Hence $\pr(\Gamma(S^1))$ is not contained in $\interior(\pr(B)) = \{ q \in \R^n \mid |q|<1\}$. 
\end{rem} 

\begin{lem}\label{lem_nice_convex_body} 
When $n \ge 2$, for any convex body $K \subset T^*\R^n$, 
there exists a sequence of nice convex bodies 
which converges to $K$ in the Hausdorff distance. 
\end{lem}
\begin{proof} 
It is easy to see that 
there exists a sequence $(K_j)_j$
such that each $\partial K_j$ is of $C^\infty$ and strictly convex, 
and $\lim_{j \to \infty} K_j = K$ in the Hausdorff distance. 
Thus it is sufficient to show that, 
for any convex body $C \subset T^*\R^n$ 
such that $\partial C$ is of $C^\infty$ and strictly convex, 
there exists $C'$ which is nice and arbitrarily close to $C$. 
Since $C$ is strictly convex, 
\[ 
L_C: =\{ x \in \partial C \mid \pr(x) \in \partial (\pr (C))\}
\] 
is a submanifold of $\partial C$ which is diffeomorphic to $S^{n-1}$, 
in particular its codimension in $\partial C$ is $n$. 
Since $n \ge 2$, there exists $C'$ which is arbitrarily $C^\infty$-close to $C$, 
and all closed characteristics of $\partial C'$ are disjoint from $L_{C'}$, 
which implies that $C'$ is nice. 
\end{proof} 

By Lemma \ref{lem_nice_convex_body}, 
Theorem \ref{SH=EHZ} is reduced to the following theorem:

\begin{thm}\label{SH=EHZ_refined}
For any $n \in \Z_{\ge 2}$ 
and any nice convex body $K \subset T^*\R^n$, 
there holds $c_{\SH}(K)=c_{\EHZ}(K)$. 
\end{thm} 

In the rest of this section we prove Theorem \ref{SH=EHZ_refined}. 
Let $n \in \Z_{\ge 2}$ and $K$ be any nice convex body in $T^*\R^n$. 
Let $\Gamma$ be a nice curve on $\partial K$, 
and 
$\gamma:= \pr \circ \Gamma: S^1 \to \interior(\pr(K))$. 
By Lemma \ref{properties_of_len_K} (iii), there holds 
\[ 
\len_K(\gamma) = \int_{S^1} \Gamma^*\bigg(\sum_i p_i dq_i \bigg) = c_{\EHZ}(K). 
\]

\begin{lem}\label{d_t_gamma}
$\dot{\gamma}(t) \ne 0$ for any  $t \in S^1$. 
\end{lem}
\begin{proof} 
Let $\nu$ be a unit vector which is normal to $T_{\Gamma(t)}(\partial K)$. 
Since $\dot{\Gamma}(t)$ is parallel to $J_{\std}(\nu)$, 
it is sufficient to show that the $p$-component of $\nu$ is nonzero. 
If the $p$-component of $\nu$ is zero, 
then the convexity of $K$ implies 
$(q,p) \in K \implies q \cdot \nu \le \gamma(t) \cdot \nu$, 
thus $\gamma(t) \in \partial(\pr(K))$, 
which contradicts the assumption 
$\gamma(S^1) \subset \interior (\pr(K))$. 
\end{proof} 

\begin{lem}\label{critical_pt_of_length} 
Let $(\gamma_s)_{-1 \le s \le 1}$ be a $C^\infty$-family of elements of $C^\infty(S^1, \interior(\pr(K)))$ 
such that $\gamma_0=\gamma$. 
Then $\frac{d}{ds} \bigg( \len_K(\gamma_s) \bigg)_{s=0} = 0$. 
\end{lem} 
\begin{proof} 
Since $\dot{\gamma}(t) \ne 0$ for any $t \in S^1$, 
we may assume that $\dot{\gamma}_s(t) \ne 0$ 
for any $(s,t) \in [-1,1] \times S^1$. 
Let us define
$\bar{\gamma}_s: S^1 \to \partial K$ 
as in Lemma \ref{properties_of_len_K} (iii). 
Namely, 
\[ 
\bar{\gamma}_s(t) = (\gamma_s(t),  p_{\gamma_s}(t)), \qquad
p_{\gamma_s}(t)  \cdot  \dot{\gamma}_s(t) = \max_{p \in K_{\gamma_s(t)}}   p \cdot \dot{\gamma}_s(t). 
\] 
Then $\Gamma=\bar{\gamma}_0$, 
and 
$\len_K(\gamma_s) = \int_{S^1} (\bar{\gamma}_s)^*\bigg(\sum_i p_i dq_i \bigg)$
for every $s \in [-1, 1]$. 
Thus
\[ 
\frac{d}{ds}\bigg(\len_K(\gamma_s)\bigg)_{s=0} 
= 
\frac{d}{ds}\bigg(\int_{S^1} (\bar{\gamma}_s)^*\bigg(\sum_i p_i dq_i \bigg) \bigg)_{s=0}
=
\int_{S^1} \omega_n((\partial_s \bar{\gamma}_s)_{s=0}(t), \, \dot{\Gamma}(t)) \, dt =0. 
\] 
\end{proof}

For any $a \in \R_{\ge 0}$ and $x \in \R^n$, 
let us define 
$\gamma_{a,x} \in \Lambda$ by 
$\gamma_{a,x}(t):= a \gamma(t) + x$. 
Let 
\[ 
T:= \{ (a,x) \in \R_{\ge 0} \times \R^n \mid \gamma_{a,x}(S^1) \subset \pr(K)  \}. 
\] 
It is easy to see that $T$ is a compact convex set in $\R_{\ge 0} \times \R^n$. 
Let us define a function $L: T \to \R$
by
$L(a,x):= \len_K(\gamma_{a,x})$. 
Obviously $L(1,0,\ldots, 0) = \len_K(\gamma)=c_{\EHZ}(K)$. 
By Lemma \ref{properties_of_len_K} (iv), 
$L$ is continuous. 

\begin{lem}\label{upper_bound_of_length}
$L(a,x) \le L(1,0, \ldots, 0)$ for any 
$(a,x) \in T$. 
\end{lem}
\begin{proof}
By the continuity of $L$,
it is sufficient to prove the lemma for $(a,x) \in \interior T$. 
For any $s \in [0,1]$, let 
\[ 
\gamma_s:=\gamma_{sa+(1-s),sx}, \quad
L_s:=\len_K(\gamma_s):=L(sa+(1-s), sx)
\] 
Our goal is to prove $L_1 \le L_0$. 

For any $s \in [0,1]$, 
we have $(sa + (1-s), sx) \in \interior T$. 
This implies that 
$\gamma_s(S^1) \subset \interior(\pr(K))$ 
and 
$sa+(1-s)>0$, 
thus 
$\dot{\gamma}_s(t) = (sa+(1-s))\dot{\gamma}(t)\ne 0$ for any $t\in S^1$. 
Let us abbreviate $p_{\gamma_s}$ as $p_s$. 
Then 
\[ 
L_s= \int_{S^1} p_s(t) \cdot \dot{\gamma}_s(t) \, dt. 
\] 
By $(\gamma_0(t), p_0(t)), (\gamma_1(t), p_1(t)) \in K$ 
and the convexity of $K$, 
\[
(\gamma_s(t), (1-s)p_0(t) + sp_1(t)) \in K. 
\] 
Then 
\[ 
p_s(t) \cdot \dot{\gamma}_s(t) 
= \max_{p \in K_{\gamma_s(t)}} p \cdot \dot{\gamma}_s(t) 
\ge ((1-s) p_0(t) + s p_1(t)) \cdot  \dot{\gamma}_s(t). 
\] 
On the other hand 
$\dot{\gamma}_s(t) = (sa+(1-s)) \dot{\gamma}(t)$, thus 
\[ 
L_s \ge \int_{S^1}  ( 1 + (a-1) s) \dot{\gamma}(t) \cdot  (p_0(t) + (p_1(t) - p_0(t))s ) \, dt
\] 
and the equality holds for $s=0$. 
Hence
\[
\partial_s L_s|_{s=0} \ge \int_{S^1}  \dot{\gamma}(t) \cdot ( (a-2) p_0(t) + p_1(t) ) \, dt. 
\] 
On the other hand 
$\partial_s L_s|_{s=0} = 0$ by Lemma \ref{critical_pt_of_length}. 
Then we obtain 
\[
\int_{S^1} \dot{\gamma}(t)  \cdot p_1(t) \, dt \le (2-a) \int_{S^1} \dot{\gamma}(t)  \cdot p_0(t) \, dt. 
\] 
Now we can finish the proof by 
\[ 
L_1 - L_0 = \int_{S^1}  a \dot{\gamma}(t)   \cdot p_1(t) - \dot{\gamma}(t) \cdot p_0(t) \, dt \le - (a-1)^2 L_0 \le 0. 
\] 
The first inequality follows from $a \ge 0$, 
and the second inequality 
follows from $L_0 \ge 0$, 
which is obvious since 
$L_0 = \len_K(\gamma) = c_{\EHZ}(K)>0$. 
\end{proof} 

We have proved 
\[ 
\max_{(a,x) \in T} \len_K(\gamma_{a,x})
= \len_K(\gamma) = c_\EHZ(K). 
\] 
On the other hand, if $(a,x) \notin T$, 
then $\len_K(\gamma_{a,x}) = - \infty$. 
Thus for any $C>c_\EHZ(K)$, 
one can define a map 
\[ 
\ell^C: (\R_{\ge 0} \times \R^n, \R_{\ge 0} \times \R^n \setminus T) \to (\Lambda^C_K, \Lambda^0_K); \quad (a, x) \mapsto \gamma_{a,x}. 
\] 
Now consider the commutative diagram
\[ 
\xymatrix{
H_n(\R^n, \R^n \setminus \pr(K))  \ar[r]^-{H_n(j^C_K)}\ar[d]& H_n(\Lambda^C_K, \Lambda^0_K )\\ 
H_n(\R_{\ge 0} \times \R^n, \R_{\ge 0} \times \R^n \setminus T) \ar[ru]_-{H_n(\ell^C)} &
}
\] 
where the vertical map is induced by the map $q \mapsto (0,q)$. 
Since $T$ is bounded, the vertical map is $0$. 
Then $H_*(j^C_K)=0$, 
which implies 
$c_{\SH}(K) \le C$. 
Since $C$ is any number larger than $c_\EHZ(K)$, 
we obtain $c_{\SH}(K) \le c_{\EHZ}(K)$. 
The inverse  inequality 
$c_{\SH}(K) \ge c_{\EHZ}(K)$ 
follows from Proposition \ref{properties_of_c_SH} (iii), 
thus we have proved Theorem \ref{SH=EHZ_refined}, 
to which Theorem \ref{SH=EHZ} was reduced. 
\qed 

\section{Proof of Theorem \ref{HZ_subadditivity}} 

The goal of this section is to prove Theorem \ref{HZ_subadditivity}. 
Let us recall the situation: 
$K$ is a compact set in $T^*\R^n$ with $\interior (K) \ne \emptyset$, 
$\Pi$ is a hyperplane which intersects $\interior (K)$, 
$\Pi^+$ and $\Pi^-$ are distinct closed halfspaces with $\partial \Pi^+= \partial \Pi^- =\Pi$, 
and $K^+:= K \cap \Pi^+$, $K^- = K \cap \Pi^-$. 
Then our goal is to prove 
\[ 
c_{\HZ}(K) \le c_{\EHZ}(\conv (K^+)) + c_{\EHZ}(\conv (K^-)), 
\] 
where $\conv$ denotes the convex hull. 

Let $K':= \conv(K^+) \cup \conv(K^-)$. 
Then $K'$ is star-shaped, thus it is a RCT set. 
We first need the following lemma: 

\begin{lem}\label{HZ_SH}
If $C \subset T^*\R^n$ is a RCT set satisfying $\interior (C) \ne \emptyset$, 
then $c_{\HZ}(C) \le c_{\SH}(C)$. 
\end{lem} 
\begin{proof} 
First we need to recall 
Corollary 3.5 of \cite{Irie_HZ}: 
for any $2n$-dimensional Liouville domain $(W, \lambda)$ and $a \in \R_{>0} \setminus \mathrm{Spec}(W, \lambda)$ 
such that the canonical map $\iota_a: H^{n-*}(W) \to \HF^{<a}_*(W,  \lambda)$ 
satisfies $\iota_a(1)=0$, there holds 
$c_{\HZ}(\interior W,  d\lambda) \le a$. 
Moreover, since $\mathrm{Spec}(W, \lambda)$ is a measure zero set, 
the assumption $a \notin \mathrm{Spec}(W, \lambda)$ can be omitted. 

Now let us assume that $C \subset T^*\R^n$ is a $C^\infty$-RCT set with a nice action spectrum in the sense of \cite{Hermann}. 
There exists $X \in \mca{X}(T^*\R^n)$ satisfying $L_X \omega_n \equiv \omega_n$ and $X$ points outwards on $\partial C$. 
Setting $\lambda:= (i_X \omega_n)|_C$, 
$(C, \lambda)$ is a Liouville domain and 
there exists a canonical isomorphism $\HF^{<a}_*(C, \lambda) \cong \SH^{[0,a)}_*(C)$ such that $\iota_a$ corresponds to $i^a_C$ 
(see Section 4, in particular Proposition 4.5 of \cite{Hermann}). 
Now, if $a> c_{\SH}(C)$ then $\iota_a (1)=0$, thus $c_{\HZ}(C) \le a$. 
This completes the proof when $C$ is a $C^\infty$-RCT set with a nice action spectrum. 

Let $C$ be an arbitrary RCT set in $T^*\R^n$. 
Then there exists a sequence of $C^\infty$-RCT sets (with nice action spectra) $(C_j)_{j \ge 1}$ 
such that $C_{j+1} \subset C_j$ for every $j \ge 1$ and $\bigcap_{j=1}^\infty C_j = C$. 
Then $\SH^{[0,a)}_*(C) \cong \varinjlim_{j \to \infty} \SH^{[0,a)}_*(C_j)$ for every $a>0$, 
which implies $c_{\SH}(C) = \lim_{j \to \infty} c_{\SH}(C_j)$. 
On the other hand, for each $j$ there holds 
$c_{\HZ}(C) \le c_{\HZ}(C_j) \le c_{\SH}(C_j)$
thus we obtain 
$c_{\HZ}(C) \le \lim_{j \to \infty} c_{\SH}(C_j) =c_{\SH}(C)$. 
\end{proof} 

Now let us state the key inequality: 

\begin{lem}\label{SH_EHZ_EHZ}
$c_{\SH}(K') \le c_{\EHZ}(\conv(K^+)) + c_{\EHZ}(\conv(K^-))$. 
\end{lem} 

Assuming Lemma \ref{SH_EHZ_EHZ}, we obtain 
\[ 
c_{\HZ}(K) \le c_{\HZ}(K') \le c_{\SH}(K') \le c_{\EHZ}(\conv(K^+)) + c_{\EHZ}(\conv(K^-)), 
\] 
where the first inequality follows from $K \subset K'$, 
the second inequality follows from Lemma \ref{HZ_SH}, and 
the last inequality is Lemma \ref{SH_EHZ_EHZ}. 
Hence we have reduced Theorem \ref{HZ_subadditivity} to Lemma \ref{SH_EHZ_EHZ}.

\subsection{Proof of Lemma \ref{SH_EHZ_EHZ}}

The case $n=1$ is easy to prove. 
Indeed, for any compact $S \subset T^*\R^1$ satisfying $\interior (S) \ne \emptyset$, 
there holds $c_{\HZ}(S) \le |S|$, where $| \, \cdot \,|$ denotes the measure. 
Also, $|S|=c_{\EHZ}(S)$ if $S$ is convex. 
Then we can prove the case $n=1$ by 
\[ 
c_{\HZ}(K') \le |K'| = |\conv(K^+)| + |\conv(K^-)| = c_{\EHZ}(\conv(K^+)) + c_{\EHZ}(\conv(K^-)). 
\]

Hence in the rest of the proof we may assume $n \ge 2$. 
We may also  assume that $\Pi=\{q_1=0\}$, 
since for any hyperplane $\Pi$  
there exists an affine map $A$ on $T^*\R^n$
with $A^*\omega_n=\omega_n$
and $A(\Pi)= \{q_1=0\}$. 
Finally, we assume that 
$K^+ = K \cap \{q_1 \ge 0\}$, 
$K^- = K \cap \{q_1 \le 0\}$. 

\begin{lem}
$K'$ is fiberwise convex.
\end{lem}
\begin{proof}
Let $q=(q_1, \ldots, q_n) \in \R^n$. 
If $q_1>0$, then $K'_q = K' \cap T^*_q\R^n = \conv (K^+) \cap T^*_q\R^n$, 
thus $K'_q$ is convex. 
Similarly, if $q_1<0$, then $K'_q = \conv(K^-) \cap T^*_q\R^n$, thus $K'_q$ is convex. 
Finally, when $q_1=0$, 
there holds $K'_q = \conv(K^+) \cap T^*_q\R^n = \conv(K^-) \cap T^*_q\R^n$, 
since 
$\conv(K^+) \cap \{q_1=0\} = \conv(K \cap \{q_1=0\}) = \conv(K^-) \cap \{q_1=0\}$. 
In particular, $K'_q$ is convex.
\end{proof}

For any $A \in \R_{>0}$, let us consider the map 
$j^A_{K'}: (\R^n, \R^n \setminus \pr(K')) \to (\Lambda^A_{K'}, \Lambda^0_{K'})$ 
which maps each $q \in \R^n$ to the constant loop at $q$. 
By Corollary \ref{c_SH_and_loop_space_homology}, 
to prove Lemma \ref{SH_EHZ_EHZ} 
it is sufficient to prove the following: 
\begin{equation}\label{AK+K-}
A  > c_{\EHZ}(\conv(K^+)) + c_{\EHZ}(\conv(K^-))  \implies  H_n(j^A_{K'})=0. 
\end{equation}
By Lemma \ref{lem_nice_convex_body}, 
there exist nice convex bodies $C^+$ and $C^-$ 
such that 
$\conv (K^+) \subset C^+$, 
$\conv (K^-) \subset C^-$ and 
$c_{\EHZ}(C^+) + c_{\EHZ}(C^-) <A$. 
Let $\Gamma^+: S^1 \to \partial C^+$ be a nice curve on $C^+$, 
and $\Gamma^-: S^1 \to \partial C^-$ be a nice curve on $C^-$. 
By changing parameterizations if necessary, 
we may assume that the following properties hold: 
\begin{itemize} 
\item The $q_1$-component of $\pr \circ \Gamma^+: S^1 \to \R^n$ takes its minimum at $0 \in S^1$, 
\item The $q_1$-component of $\pr \circ \Gamma^-: S^1 \to \R^n$ takes its maximum at $0 \in S^1$. 
\end{itemize} 
Then there exist
$\gamma^+: S^1 \to \R_{\ge 0} \times \R^{n-1}$
and 
$\gamma^-: S^1 \to \R_{\le 0} \times \R^{n-1}$
such that 
$\gamma^+ - \pr \circ \Gamma^+$ and 
$\gamma^-  - \pr \circ \Gamma^-$ 
are constant maps from $S^1$ to $\R^n$.

\begin{rem}\label{gamma_pm_nonconstant}
By Lemma \ref{d_t_gamma}, 
$\gamma^+$ and $\gamma^-$ are nonconstant. 
\end{rem} 

\begin{lem}\label{upper_bound_of_length_2} 
\begin{enumerate} 
\item[(i):] 
For any $a \in \R_{\ge 0}$ and $x \in \R_{\ge 0} \times \R^{n-1}$, 
\[ 
\gamma^+_{a,x}: S^1 \to \R^n; \, t \mapsto a \gamma^+(t) + x
\] 
satisfies $\len_{K'} (\gamma^+_{a,x}) \le c_{\EHZ}(C^+)$. 
\item[(ii):] 
For any $a \in \R_{\ge 0}$ and $x \in \R_{\le 0} \times \R^{n-1}$, 
\[ 
\gamma^-_{a,x}: S^1 \to \R^n; \, t \mapsto a \gamma^-(t) + x
\] 
satisfies $\len_{K'} (\gamma^-_{a,x}) \le c_{\EHZ}(C^-)$. 
\end{enumerate} 
\end{lem}
\begin{proof}
Since $\gamma^+_{a,x}(S^1) \subset \R_{\ge 0} \times \R^{n-1}$
and $K' \cap \pr^{-1}(\R_{\ge 0} \times \R^{n-1}) \subset C^+$, 
there holds $\len_{K'} (\gamma^+_{a,x}) \le \len_{C^+}(\gamma^+_{a,x})$. 
On the other hand, Lemma \ref{upper_bound_of_length} 
implies $\len_{C^+}(\gamma^+_{a,x}) \le c_{\EHZ}(C^+)$, 
which completes the proof of (i). 
The proof of (ii) is similar to the proof of (i). 
\end{proof}

For any $(s, t, x_2, \ldots, x_n) \in (\R^2 \setminus (\R_{<0})^2) \times \R^{n-1}$, 
we define 
$\gamma_{s,t,x_2,\ldots, x_n}: S^1 \to \R^n$ as follows: 
\begin{itemize} 
\item When $s \le 0$ and $t \ge 0$, 
\[ 
\gamma_{s,t,x_2,\ldots, x_n}(\theta) := \begin{cases} t \cdot \gamma^+(2\theta) + (-s, x_2, \ldots, x_n)  &(0 \le \theta \le 1/2) \\  (-s, x_2, \ldots, x_n) &(1/2 \le \theta \le 1).  \end{cases}
\] 
\item When $s,t \ge 0$, 
\[
\gamma_{s,t,x_2,\ldots,x_n}(\theta) := \begin{cases}  t \cdot \gamma^+(2\theta) + (0,x_2, \ldots, x_n) &(0 \le \theta \le 1/2) \\  s \cdot \gamma^-(2\theta-1) + (0,x_2, \ldots, x_n) &(1/2 \le \theta \le 1). \end{cases}
\] 
\item When $s \ge 0$ and $t \le 0$, 
\[ 
\gamma_{s,t,x_2,\ldots, x_n}(\theta):= \begin{cases}   (t,x_2, \ldots, x_n) &(0 \le \theta \le 1/2) \\ s \cdot \gamma^-(2\theta-1) + (t,x_2,\ldots,x_n) &(1/2 \le \theta \le 1).\end{cases}
\]
\end{itemize} 
Then, Lemma \ref{upper_bound_of_length_2} implies 
\[ 
\sup_{(s,t,x_2, \ldots, x_n) \in (\R^2 \setminus (\R_{<0})^2) \times \R^{n-1}}   \len_{K'}(\gamma_{s,t,x_2, \ldots, x_n}) \le c_{\EHZ}(C^+) + c_{\EHZ}(C^-) < A, 
\] 
thus one can define a map 
\[ 
\ell^A: (\R^2 \setminus (\R_{<0})^2) \times \R^{n-1} \to \Lambda^A_{K'}; \, (s,t,x_2, \ldots, x_n) \mapsto \gamma_{s,t,x_2,\ldots,x_n}.
\] 
It is easy to check that $\ell^A$ is continuous with respect to the $L^{1,2}$-topology on $\Lambda$. 
For any $(x_1, \ldots, x_n) \in \R^n$, let $c_{(x_1, \ldots, x_n)}$ denote the constant map from $S^1$ to $(x_1, \ldots, x_n)$. 

\begin{lem}\label{lem_ell_A}
\begin{enumerate}
\item[(i):] For any $r \in \R_{\le 0}$, 
\[ 
\gamma_{r,0,x_2, \ldots,x_n} = c_{(-r, x_2, \ldots, x_n)}, \qquad
\gamma_{0,r,x_2,, \ldots,x_n} = c_{(r, x_2,\ldots, x_n)}. 
\] 
\item[(ii):] There exists $R \in \R_{>0}$ such that 
\[
\max\{ |s|, |t|, |(x_2,\ldots, x_n)| \} > R \implies \len_{K'}(\gamma_{s,t,x_2,\ldots,x_n})  = -\infty. 
\]
\end{enumerate}
\end{lem}
\begin{proof} 
(i) follows directly from the definition. 
To prove (ii), let us take $R>0$ so that the following conditions hold: 
\[ 
B^n(R) \supset \pr(K'), \qquad 
R \cdot \min\{ \diam(\gamma^+(S^1)), \, \diam(\gamma^-(S^1))\} \ge  \diam(\pr(K')). 
\] 
Here $B^n(R):=\{q \in \R^n \mid |q| \le R\}$ 
and $\diam$ denotes the diameter. 
Note that the second condition can be achieved
when $R$ is sufficiently large, 
since
$\gamma^+$ and $\gamma^-$ are 
both nonconstant maps (see Remark \ref{gamma_pm_nonconstant}). 

Let us prove that such $R$ satisfies the required conditions: 
if $\len_{K'}(\gamma_{s,t,x_2,\ldots,x_n})>-\infty$ (which is equivalent to $\gamma_{s,t,x_2,\ldots,x_n}(S^1) \subset \pr(K')$)
then $\max\{ |s|, |t|, |(x_2,\ldots,x_n)|\} \le R$. 
It is sufficient to consider the following three cases: 
\begin{itemize}
\item $s \le 0$ and $t \ge 0$ :
Since $t \cdot \diam(\gamma^+(S^1)) \le \diam (\pr(K'))$, we obtain $t \le R$. 
Since $\gamma_{s,t,x_2,\ldots,x_n}(0) = (-s,x_2,\ldots, x_n) \in \pr(K') \subset B^n(R)$, 
we obtain $|s|, |(x_2,\ldots, x_n)| \le R$. 
\item $s,t \ge 0$ :
Since $t \cdot \diam(\gamma^+(S^1)), s \cdot \diam(\gamma^-(S^1)) \le \diam(\pr(K'))$, we obtain $t, s \le R$. 
Since $\gamma_{s,t,x_2,\ldots,x_n}(0)=(0,x_2,\ldots,x_n) \in \pr(K') \subset B^n(R)$, 
we obtain $|(x_2,\ldots,x_n)| \le R$. 
\item $s \ge 0$ and $t \le 0$ :
this case is similar to the first case. 
\end{itemize} 
\end{proof} 

Let us define $h: \R^n \to (\R^2 \setminus (\R_{<0})^2) \times \R^{n-1}$ by 
\[ 
h(x_1, x_2, \ldots, x_n) = \begin{cases} (-x_1, 0, x_2, \ldots, x_n) &(x_1 \ge 0),  \\ (0, x_1, x_2, \ldots, x_n) &(x_1 \le 0). \end{cases}
\] 
Then Lemma \ref{lem_ell_A} (i) implies
$\ell^A \circ h(x_1, \ldots, x_n) = c_{(x_1,\ldots, x_n)}$. 
By Lemma \ref{lem_ell_A} (ii), 
when $R \in \R_{>0}$ is sufficiently large,
\[ 
H_n(\ell^A \circ h):  H_n(\R^n, \R^n \setminus B^n(R)) \to H_n(\Lambda^A_{K'}, \Lambda^{-\infty}_{K'}) \to H_n(\Lambda^A_{K'}, \Lambda^0_{K'})
\] 
is zero. We may also assume that $\pr(K') \subset B^n(R)$. 
Now the diagram 
\[ 
\xymatrix{
H_n(\R^n, \R^n \setminus \pr(K'))  \ar[r]^-{H_n(j^A_{K'})} & H_n(\Lambda^A_{K'}, \Lambda^0_{K'}) \\
H_n(\R^n, \R^n \setminus B^n(R))  \ar[u]\ar[ru]_-{H_n(\ell^A \circ h)} &
}
\] 
commutes, 
the vertical map is surjective
(since $\pr(K')$ is star-shaped) 
and the diagonal map is zero, 
thus $H_n(j^A_{K'})=0$, which completes the proof of (\ref{AK+K-}). 
\qed

\section{Proof of Proposition \ref{2021_0530_1}}

First let us introduce a few notations. 
For any $S \subset \R^n$, let 
\begin{align*} 
D^*S&:= \{ (q,p) \in T^*\R^n \mid q \in S, \, |p| \le 1\}, \\ 
w(S)&:= \inf \{ \sup h - \inf h \mid h \in C^\infty_c(\R^n), \text{ $|dh(x)| \ge 1$ for any $x \in S$} \},  \\ 
r(S)&:= \sup \{ r \mid \text{there exists $q \in \R^n$ with $B^n(q:r) \subset S$} \}. 
\end{align*} 
$B^n(q:r)$ denotes the closed ball in $\R^n$ with center $q$ and radius $r$. 

Our goal is to show that, for any bounded $B \subset T^*\R^n$ 
and any $\ep \in \R_{>0}$, there exist compact star-shaped sets $K_1, K_2 \subset T^*\R^n$
such that $B \subset K_1 \cup K_2$ and $e(K_1), e(K_2) < \ep$. 
Note that, for any compact $K \subset T^*\R^n$ and $a>0$, there holds 
$e(aK)= a^2 e(K)$. Thus we may assume that 
$B$ is a subset of $D^*B^n(1)= \{(q,p) \in T^*\R^n \mid |q|, |p| \le 1\}$. 

For any nonempty compact $S \subset \R^n$, there holds 
\[ 
e(D^*S) \le 2 w(S) \le C_n  r(S)
\] 
where $C_n$ is a positive constant which depends only on $n$. 
The first inequality is proved in Lemma 4 of \cite{Irie_displacement}, 
and the second inequality is proved in Section 2.2 of \cite{Irie_displacement}, 
although notations and settings in this section are slightly different from those in \cite{Irie_displacement}.

For any $\theta$, let $R_\theta$ denote the anti-clockwise rotation of $\R^2$ with 
center $(0,0)$ and angle $\theta$. 
For any integer $N \ge 1$, let 
\[ 
T(N):= \{ (r \cos \theta, r \sin \theta) \mid 0 \le r \le 1, \, 0 \le \theta \le \pi/N\} \subset \R^2. 
\] 
Moreover, for any $i \in \{1, 2\}$, let 
us define $S_i(N) \subset \R^2$ and $\bar{S}_i(N) \subset \R^n$ by 
\[ 
S_i(N):= \bigcup_{j=0}^{N-1}  R_{\frac{(i+2j-1)\pi}{N}}(T(N)), \qquad 
\bar{S}_i(N):=\begin{cases}  S_i(N) &(n=2)  \\  S_i(N) \times B^{n-2}(1) &(n \ge 3). \end{cases} 
\]
Then 
$D^*\bar{S}_i(N) \subset T^*\R^n$ is a compact star-shaped set for any $i \in \{1,2\}$, and there holds
\[ 
B \subset D^* B^n(1) \subset D^* \bar{S}_1(N) \cup D^* \bar{S}_2(N). 
\] 
On the other hand, for any $N$ and $i$, 
\[ 
r(\bar{S}_i(N)) \le r(S_i(N)) \le r(T(N)) \le \frac{\pi}{2N}. 
\] 
Thus, if $N  > \frac{\pi C_n}{2\ep}$, then 
$\max_{1 \le i \le 2} e( D^* \bar{S}_i(N)) < \ep$. 
One can complete the proof by taking such $N$ and setting 
$K_i:= D^*\bar{S}_i(N) \, (i=1, 2)$. 
\qed

\end{document}